\documentclass[11pt,a4paper]{article}
\usepackage[T1]{fontenc}
\usepackage{lmodern}
\usepackage{amsmath,amssymb,amsthm,mathtools}
\mathtoolsset{showonlyrefs} 
\usepackage[margin=28mm,headheight=15pt]{geometry}
\usepackage{microtype,enumitem,booktabs,longtable,array,xcolor}
\usepackage[colorlinks=true,linkcolor=blue!45!black,citecolor=blue!45!black,urlcolor=blue!45!black]{hyperref}
\hypersetup{pdftitle={A priori Dirichlet energy bounds for positive Trudinger--Moser critical points on simply connected planar domains}}
\setlist[enumerate]{itemsep=0.35em,topsep=0.4em}
\setlist[itemize]{itemsep=0.3em,topsep=0.4em}
\numberwithin{equation}{section}
\newtheorem{theorem}{Theorem}[section]
\newtheorem{proposition}[theorem]{Proposition}
\newtheorem{lemma}[theorem]{Lemma}

\newtheorem{tool}[theorem]{Theorem}
\theoremstyle{definition}

\theoremstyle{remark}
\newtheorem{remark}[theorem]{Remark}
\DeclareMathOperator{\diver}{div}
\DeclareMathOperator{\supp}{supp}
\DeclareMathOperator{\dist}{dist}

\DeclareMathOperator{\Lip}{Lip}
\DeclareMathOperator{\TV}{TV}
\newcommand{\R}{\mathbb R}
\newcommand{\D}{\mathbb D}

\newcommand{\dd}{\,\mathrm d}
\newcommand{\ind}{\mathbf 1}
\newcommand{\weak}{\rightharpoonup}
\newcommand{\eps}{\varepsilon}
\newcommand{\Om}{\Omega}
\newcommand{\cl}{\overline}

\newcommand{\rev}[1]{{#1}}
\newenvironment{revision}{\par\begingroup}{\par\endgroup}
\allowdisplaybreaks[2]
\title{\bfseries A priori Dirichlet energy bounds for positive Trudinger--Moser critical points on simply connected planar domains}
\author{%
Hongtao Hu\thanks{\href{mailto:huht23@mails.tsinghua.edu.cn}{\texttt{huht23@mails.tsinghua.edu.cn}}}
\qquad Qiaoran Wu\thanks{\href{mailto:wuqr24@mails.tsinghua.edu.cn}{\texttt{wuqr24@mails.tsinghua.edu.cn}}}
\qquad Wenming Zou\thanks{\href{mailto:zou-wm@math.tsinghua.edu.cn}{\texttt{zou-wm@math.tsinghua.edu.cn}}}\\[0.5em]
{\small\rev{Department of Mathematical Sciences, Tsinghua University}}\\
{\small\rev{Beijing 100084, China}}}
\date{}

\begin{document}
\maketitle
\begin{abstract}
\begin{revision}
\noindent Let \(\Om\subset\R^2\) be a smooth bounded simply connected domain. We prove that all positive solutions of \(-\Delta u=\theta ue^{u^2}\), with \(\theta>0\) and zero Dirichlet boundary data, have Dirichlet energy bounded by a constant depending only on \(\Om\). Consequently, the Trudinger--Moser functional has no nonnegative constrained critical points at sufficiently large prescribed energies. The proof combines a weighted Pohozaev identity obtained after conformal change of variables, concentration analysis of normalized measures, and local radial comparison. The comparison hypotheses are obtained from local integral estimates, without an a priori energy bound. An estimate for the exterior Green potential then excludes sequences of unbounded energy.
\end{revision}
\end{abstract}



\section{Introduction}

Let \(\Om\subset\R^2\) be a smooth bounded domain. In dimension two, the Sobolev space \(H^1_0(\Om)\) embeds into every \(L^q(\Om)\), \(1\le q<\infty\), but not into \(L^\infty(\Om)\). \rev{The Trudinger--Moser inequality~\cite{Moser,Trudinger} describes the corresponding critical exponential integrability.} Set
\[
 F(u):=\int_\Om(e^{u^2}-1)\dd x,\qquad
 \mathcal M_\beta:=\left\{u\in H^1_0(\Om):\int_\Om|\nabla u|^2\dd x=\beta\right\},
 \quad \beta>0.
\]
Moser~\cite{Moser} identified the sharp constant \(4\pi\):
\[
 \sup_{u\in\mathcal M_\beta}F(u)<\infty
 \quad\Longleftrightarrow\quad \beta\le4\pi.
\]
For \(0<\beta<4\pi\), compactness ensures that the supremum is attained. At the critical energy \(\beta=4\pi\), the embedding is no longer compact. Carleson--Chang~\cite{CC} first proved the existence of an extremal on the disk, and Flucher~\cite{Flucher} extended this result to general bounded planar domains. For \(\beta>4\pi\), the restriction \(F|_{\mathcal M_\beta}\) is unbounded above, although positive constrained critical points may still exist.

\begin{revision}
Under a compactness assumption on the domain, Struwe~\cite[Theorems~1.7 and~1.8]{Struwe1988} obtained local maximizers at energies slightly above \(4\pi\), together with a second critical point for almost every nearby energy. Using a constraint-preserving heat flow and its energy quantization, Lamm--Robert--Struwe~\cite[Theorem~6.5]{LRS} subsequently proved that, for every smooth bounded planar domain, there is an \(\eps_\Om>0\) such that at least two positive constrained critical points exist at each \(\beta\in(4\pi,4\pi+\eps_\Om)\). Thus \(4\pi\) is the threshold for boundedness of the functional, but is not an upper limit for the existence of positive critical points.

At higher energies, the topology of the domain plays an important role. Using radial symmetry and blow-up analysis, Malchiodi--Martinazzi~\cite{MM} proved a uniform Dirichlet energy bound for positive constrained critical points on the disk. For general planar domains, Malchiodi--Martinazzi--Thizy~\cite[Theorem~1.2 with \(p=2\)]{MMT} computed the total Leray--Schauder degree of the corresponding problem:
\[
 d_\beta(\Om)=\binom{N-\chi(\Om)}{N},
 \qquad \beta\in(4\pi N,4\pi(N+1)],\quad N\in\{0,1,\ldots\},
\]
where \(\chi(\Om)\) is the Euler characteristic. When the domain is not simply connected, this degree is nonzero for every \(\beta>0\), yielding a positive critical point at each positive energy. On simply connected domains, the degree vanishes for \(\beta>4\pi\), but a vanishing degree does not by itself imply nonexistence.

This led to the question in~\cite{MMT}: does every fixed smooth bounded simply connected planar domain admit a threshold \(\beta^\sharp>4\pi\) above which no positive constrained critical point exists? The question is also stated explicitly in Malchiodi's survey~\cite{Malchiodi2026}. It asks whether the energy bound on the disk extends to general simply connected domains without radial symmetry.
\end{revision}

We prove this uniform energy bound. A nonzero nonnegative constrained critical point satisfies
\begin{equation}\label{eq:pde}
 -\Delta u=\theta ue^{u^2}\quad\hbox{in}\ \Om,
 \qquad u>0\quad\hbox{in}\ \Om,
 \qquad u=0\quad\hbox{on}\ \partial\Om,
 \qquad \theta>0.
\end{equation}
Here \(\theta\) is a multiplier that varies with the solution, and \(\Delta=\partial_1^2+\partial_2^2\).

\begin{theorem}\label{thm:main}
Let \(\Om\subset\R^2\) be a smooth bounded simply connected domain. There exists a finite constant \(B_\Om\), depending only on \(\Om\), such that every solution \(u\in C^2(\cl\Om)\) of \eqref{eq:pde} satisfies
\begin{equation}\label{eq:main}
 E(u):=\int_\Om|\nabla u|^2\dd x\le B_\Om.
\end{equation}
Consequently, \(F|_{\mathcal M_\beta}\) has no nonnegative critical point when \(\beta>B_\Om\).
\end{theorem}

The parameter used in \cite{MMT} is \(\lambda=\theta/2\), and its Laplacian has the opposite sign convention. Thus any \(\beta^\sharp>\max\{4\pi,B_\Om\}\) is an admissible threshold in the question on page~5 of \cite{MMT}.

The bound is uniform over all \(\theta>0\) and requires neither convexity nor star-shapedness. Together with the existence result of \cite{MMT}, it shows that positive constrained critical points have bounded energy on simply connected domains, whereas every positive energy level contains a positive critical point when the domain is not simply connected. We do not consider sign-changing critical points.

\begin{revision}
The compactness theory for solutions with bounded energy provides an analytic basis for these existence results. Adimurthi--Struwe~\cite{AS2000} studied compactness and concentration at isolated points for positive solutions of planar equations with critical exponential growth, and identified the {limit after rescaling} in the energy range allowing at most one blow-up point. Druet~\cite{Druet2006} extended the quantization analysis to arbitrary bounded energy levels. For equation~\eqref{eq:pde}, Druet--Thizy~\cite[Theorem~1.2]{DT} further proved that, if a sequence of positive solutions has uniformly bounded Dirichlet energy and \(\theta_k\to0\), then, after passage to a subsequence, there are finitely many distinct interior concentration points, each carrying exactly \(4\pi\) of energy, and
\[
 \int_\Om|\nabla u_k|^2\dd x\longrightarrow4\pi N,
 \qquad N\in\{1,2,\ldots\}.
\]
They also characterized the {locations and relative values of the local maxima} through a finite-dimensional system involving the Green function and its regular part. These results supply the compactness needed in critical point constructions and degree computations. Their conclusions for the full sequence assume bounded energy and therefore do not directly rule out the sequences with diverging energy that must be treated here.

Related problems with subcritical exponential growth also enter the analysis of the critical equation. Malchiodi--Martinazzi--Thizy~\cite{MMT} considered the following family, with \(\lambda>0\):
\[
 -\Delta u=p\lambda u^{p-1}e^{u^p}\quad\hbox{in}\ \Om,
 \qquad u>0\quad\hbox{in}\ \Om,\qquad u=0\quad\hbox{on}\ \partial\Om,
 \qquad 1\le p\le2.
\]
The endpoint \(p=1\) is of Liouville type, \(p=2\) corresponds to~\eqref{eq:pde}, and the growth for \(1<p<2\) is subcritical relative to \(H^1_0(\Om)\). After a suitable normalization, Malchiodi--Martinazzi--Thizy~\cite{MMT} connected this family to the Liouville equation by a homotopy to compute the degree above. Their normalization parameter \(\beta\) agrees with the Dirichlet energy only when \(p=2\); {results for large values of this parameter when} \(p<2\) must therefore be distinguished from the critical energy bound considered here. \rev{On p.~291 of~\cite{Malchiodi2026}, references~4 and~5 are cited for nonexistence at large values of the normalization parameter when \(1\le p<2\). Reference~5 is the work of Battaglia--Mancini--Thizy, listed as in preparation; the general simply connected case with \(p=2\) is stated separately.}

Other two-dimensional geometries lead to further existence and concentration results. De Marchis et al.~\cite{DMMT} studied the weighted functional \(\int_\Sigma(e^{u^2}-1)f\dd v_g\) on a closed Riemannian surface \((\Sigma,g)\), with \(f\) smooth and positive. A minmax construction for subcritical problems, compactness analysis, and the limit \(p\uparrow2\) yield positive constrained critical points at every positive \(H^1\) energy. Here the constraint uses the full \(H^1\) norm, rather than the Dirichlet integral alone. Chen et al.~\cite[Theorem~1.1]{CHLPW} studied \(-\Delta_g u_k=\lambda_k u_ke^{u_k^2}\) on the Poincar\'e disk \((\mathbb B^2,g)\), with positive solutions tending to zero at hyperbolic infinity. For a sequence with uniformly bounded Dirichlet energy and \(0<\lambda_k\to0\), they obtained, after suitable M\"obius transformations and passage to a subsequence,
\[
 \int_{\mathbb B^2}|\nabla_g u_k|_g^2\dd v_g\longrightarrow4\pi,
 \qquad |\nabla_g u_k|_g^2\dd v_g\rightharpoonup4\pi\delta_0.
\]
Here \(\delta_0\) is the unit Dirac measure at the origin: all the energy concentrates at one point, and its total amount tends to \(4\pi\). They also used an asymptotic energy expansion to obtain positive constrained critical points at energies slightly above \(4\pi\).

Critical exponential growth also occurs in higher dimensions and for higher-order operators, with different operators and energy constraints. On a smooth bounded domain \(D\subset\R^n\), \(n\ge2\), the critical growth for the first-order space \(W^{1,n}_0(D)\) is \(e^{\alpha|v|^{n/(n-1)}}\), and the associated Euler--Lagrange equations involve the \(n\)-Laplacian \(\Delta_n v=\operatorname{div}(|\nabla v|^{n-2}\nabla v)\). Under structural and growth assumptions on the nonlinearity, Adimurthi~\cite{Adimurthi1990} proved the existence of positive solutions to such critical Dirichlet problems.

On a smooth bounded domain \(D\subset\R^{2m}\), Adams~\cite{Adams} established the sharp exponential inequality for \(H^m_0(D)\): under the normalization \(\|\nabla^m v\|_{L^2(D)}\le1\), the integral \(\int_D e^{\alpha v^2}\dd x\) is uniformly bounded if and only if \(\alpha\le\beta_m:=2^{2m}\pi^m m!\), for \(\alpha>0\). Here \(\nabla^m v=\Delta^{m/2}v\) for even \(m\), and \(\nabla^m v=\nabla\Delta^{(m-1)/2}v\) for odd \(m\). At the optimal constant, Lu--Yang~\cite[Theorem~1.2]{LuYang} proved the existence of extremals in dimension four, where \(\beta_2=32\pi^2\). DelaTorre--Mancini~\cite{DelaTorreMancini} subsequently established attainment for every \(m\ge1\), including improved inequalities of Adimurthi--Druet type, using blow-up analysis of subcritical extremals and comparisons with polyharmonic functions on annular regions. These results concern constrained maximizers at the critical threshold and do not assert their positivity on general domains.

Energy quantization for higher-order equations is more directly related to the planar concentration theory. Martinazzi~\cite{Martinazzi2009} and Martinazzi--Struwe~\cite{MS} considered
\[
 (-\Delta)^m u_k=\lambda_k u_k e^{m u_k^2}\quad\hbox{in}\ D,
 \qquad u_k=\partial_\nu u_k=\cdots=\partial_\nu^{m-1}u_k=0
 \quad\hbox{on}\ \partial D.
\]
For positive solutions with uniformly bounded energy and \(0<\lambda_k\to0\), Martinazzi {proved convergence after rescaling to} \(\eta_0(x)=\log\frac{2}{1+|x|^2}\). When the limiting energy is positive, Martinazzi--Struwe further proved that
\[
 \begin{gathered}
 \lim_{k\to\infty}\int_D u_k(-\Delta)^m u_k\dd x=K\Lambda_1,
 \qquad K\in\{1,2,\ldots\},\\
 \Lambda_1=(2m-1)!|\mathbb S^{2m}|=\frac{\beta_m}{m}.
 \end{gathered}
\]
Thus the limiting total energy is a positive integer multiple of \(\Lambda_1\), which equals \(4\pi\) for \(m=1\) and \(16\pi^2\) for \(m=2\) with this normalization. This extends energy quantization to higher-order critical equations, again under an a priori bound on the energy of the sequence.
\end{revision}

We argue by contradiction. Using the notation of Section~\ref{sec:notation}, let \(M_j\) denote the total mass of the right-hand side and \(\gamma_j\) the maximum of \(u_j\). Suppose that there is a sequence of positive solutions with \(E_j\to\infty\). Integration by parts and the zero boundary condition give
\[
 E_j=\int_\Om u_j\rho_j\dd x\le\gamma_jM_j
 =\bigl(\sqrt{\theta_j}\gamma_j\bigr)
  \bigl(M_j/\sqrt{\theta_j}\bigr).
\]
We estimate the two factors separately along this sequence. Local comparison near a global maximum bounds the first. To bound the second, we study concentration of the normalized measures and estimate integrals near a suitable local maximum.

The main local tool is the radial comparison and refined radial expansion of Druet--Thizy~\cite[Proposition~3.1(a), (b) and (3.9)]{DT}, recalled in Theorem~\ref{tool:DTlocal}. Its hypotheses include convergence after rescaling, a pointwise bound for the nonlinearity, and a gradient bound. Given the rescaled convergence and the required parameter estimates, Lemma~\ref{lem:budget} derives the latter two bounds from a local weighted integral condition. We verify these assumptions at each maximum point where the comparison is used. The global blow-up theorem in \cite[Theorem~1.2]{DT} assumes bounded energy and therefore cannot be applied directly to the sequence \(E_j\to\infty\). We use its conclusions about the parameter and the maximum only in the auxiliary contradiction argument of Lemma~\ref{lem:logtheta}, where the contradiction assumption and the total mass estimate first give an energy bound.

\begin{revision}
Simple connectedness enters through the total mass estimate in Section~\ref{sec:mass}. We follow the method used by Battaglia~\cite{Battaglia2019} to obtain mass bounds for Liouville and H\'enon--Lane--Emden equations. Pulling the equation back under a conformal map \(\phi:\D\to\Om\) introduces the weight \(a=|\phi'|^2\); the Pohozaev identity for this weighted equation gives \eqref{eq:weighted-pohozaev}. For the subsequent analysis of the normalized measures, we further expand the negative outward normal derivative of the solution on the disk in a Fourier series and use Parseval's identity to rewrite the boundary term. This gives the integral identity \eqref{eq:moment} in Lemma~\ref{lem:moment}, from which we obtain \(M_j\le C_\Om\) and decay of \(M_j\) when \(\theta_j\to0\).
\end{revision}
 Thus \(E_j\to\infty\) implies \(\gamma_j\to\infty\). In Section~\ref{sec:local}, we verify the hypotheses of radial comparison at a global maximum and combine the radial expansion with the Green representation formula to prove
\[
 \theta_j\to0,\qquad \sqrt{\theta_j}\gamma_j\le C_\Om.
\]

For the second factor, we consider the normalized measures \(\nu_j,\Pi_j\) in \eqref{eq:normalizations}. The boundary strip estimate rules out boundary mass in their weak limits. Using {\rev{identity~\eqref{eq:moment}}}, the Green representation formula, and local exponential integrability, Section~\ref{sec:pressure} shows that the weak limit \(\nu\) of \(\nu_j\) has an atom of positive mass. It also determines the relation between the two limiting atomic masses and the distribution of function values near an atom. These results yield \(M_j^2\log(1/\theta_j)\to0\) and \(M_j^2E_j\to0\) in Subsection~\ref{sec:strict}, providing the parameter and oscillation estimates needed for the subsequent rescaling.

Fix an atom \(a_0\), and write \(p=\nu(\{a_0\})>0\) and \(A=4\pi/p\). In Section~\ref{sec:halfmass}, we select a ball of minimal radius near \(a_0\) with \(\nu_j\) mass \(p/2\). After rescaling, a measure classification argument and the Chen--Li theorem~\cite{CL} identify the limit. Section~\ref{sec:center} {proves local \(C^2\) convergence of the corresponding functions} and selects local maximum points \(x_j\to a_0\) of \(u_j\). Propositions~\ref{prop:representative} and \ref{prop:neighborhood} give
\[
 \beta_j:=u_j(x_j),\qquad M_j\beta_j\to A,\qquad
 \int_{B_{L_j}(x_j)}u_j\rho_j\dd x\to4\pi,
 \quad L_j=e^{-\beta_j^2/32}.
\]
Together with the rescaled convergence and the parameter estimates, this integral condition verifies all the hypotheses of Lemma~\ref{lem:budget}. Radial comparison and the formulas in Lemma~\ref{lem:matching} therefore apply at these local maximum points.

Section~\ref{sec:absorb} estimates the Green potential outside a small ball centered at \(x_j\), using radial comparison, a local Pohozaev identity, and an integral estimate with a suitable test vector field. Proposition~\ref{prop:absorb} bounds this exterior integral by \(CM_j\); {\rev{relation~\eqref{eq:matching}}} then gives \(\theta_j\beta_j^2\ge c>0\). Since \(M_j\beta_j\to A\), it follows that
\[
 \frac{M_j^2}{\theta_j}
 =\frac{(M_j\beta_j)^2}{\theta_j\beta_j^2}\le C.
\]
Combining this bound for the second factor with the global maximum estimate, we obtain
\[
 E_j\le\bigl(\sqrt{\theta_j}\gamma_j\bigr)
          \bigl(M_j/\sqrt{\theta_j}\bigr)\le C,
\]
contrary to \(E_j\to\infty\). It suffices to work along a subsequence associated with one fixed atom of positive mass: the constants along this subsequence are independent of \(j\).

\begin{revision}
The preliminary estimates needed below are collected in Section~\ref{sec:setting}. The remaining sections follow the steps of this argument, ending with the exterior potential estimate and the proof of Theorem~\ref{thm:main} in Section~\ref{sec:absorb}.
\end{revision}

\section{Preliminaries}\label{sec:setting}

\subsection{Notation, regularity, and normalized measures}\label{sec:notation}
We write \(B_r(a)=\{x\in\R^2:|x-a|<r\}\), \(B_r=B_r(0)\), \(\D=B_1(0)\), and \(S^1=\partial\D\). The vectors \(n\) and \(\tau\) are the outward unit normal vector and a unit tangent vector to the boundary under consideration, with \(\partial_n v=\nabla v\cdot n\) and \(\partial_\tau v=\nabla v\cdot\tau\). Throughout, \(\Om\subset\R^2\) is a fixed smooth bounded simply connected domain, and \(H^1_0(\Om)\) is equipped with the norm \(\|\nabla u\|_{L^2}\). The notation \(K\Subset\Om\) means that \(K\) is compactly contained in \(\Om\). Unless stated otherwise, \(L^q\) estimates are understood for \(1\le q<\infty\).
\begin{revision}
{For every fixed \(u\in H^1_0(\Om)\) and \(c>0\), the Trudinger--Moser inequality gives \(e^{cu^2}\in L^1(\Om)\). H\"older's inequality and local exponential integrability then yield \(F\in C^1(H^1_0(\Om),\R)\), with}
\[
 DF(u)[\varphi]=2\int_\Om ue^{u^2}\varphi\dd x.
\]

If \(u\in\mathcal M_\beta\), \(\beta>0\), is a constrained critical point, the Lagrange multiplier rule gives a constant \(\kappa\) such that
\[
 \int_\Om ue^{u^2}\varphi\dd x
 =\kappa\int_\Om\nabla u\cdot\nabla\varphi\dd x
 \qquad(\varphi\in H^1_0(\Om)).
\]
Testing with \(\varphi=u\) yields \(\kappa=\beta^{-1}\int_\Om u^2e^{u^2}\dd x>0\). Thus \(u\) solves \(-\Delta u=\theta ue^{u^2}\) weakly with \(\theta=\kappa^{-1}>0\). Since \(ue^{u^2}\in L^q(\Om)\) for every finite \(q\), the global \(W^{2,q}\) estimates followed by Schauder regularity give \(u\in C^2(\cl\Om)\); see \cite[Chapters~6 and~9]{GT}. If \(u\ge0\), the strong maximum principle gives \(u>0\) in \(\Om\). The same regularity argument applies to any weak \(H^1_0(\Om)\) solution of \eqref{eq:pde}.
\end{revision}

For a solution of \eqref{eq:pde}, define
\begin{equation}\label{eq:notation}
 \rho:=\theta ue^{u^2},\qquad P:=\theta(e^{u^2}-1),\qquad
 M:=\int_\Om\rho\dd x,\qquad \gamma:=\max_{\cl\Om}u.
\end{equation}
Here \(\rho\) is the nonlinear right-hand side, and \(M\) is the total mass of \(\rho\dd x\). The auxiliary function \(P\) satisfies \(P=2\int_0^u\theta t e^{t^2}\dd t\), hence \(\nabla P=2\rho\nabla u\). We abbreviate \(E=E(u)\) and shall also use the weighted integral \(\int_Bu\rho\) over a ball \(B\). Integration by parts using the zero boundary data gives
\begin{equation}\label{eq:energyidentity}
 E=\int_\Om u\rho\dd x\le\gamma M.
\end{equation}

Let \(G\) be the Dirichlet Green function of \(-\Delta\) on \(\Om\), so that, for each \(y\in\Om\),
\[
 -\Delta_xG(x,y)=\delta_y\quad\hbox{in }\Om,\qquad
 G(x,y)=0\quad\hbox{for }x\in\partial\Om,
\]
where the equation is understood in the sense of distributions and \(\delta_y\) is the unit Dirac measure at \(y\). For a finite nonnegative Borel measure \(\eta\) on \(\Om\), define its Green potential by
\[
 (G\eta)(x):=\int_\Om G(x,y)\dd\eta(y),\qquad x\in\Om.
\]
The Green potential of a general measure may take the value \(+\infty\); finiteness must be checked before using pointwise estimates. The regular part and the Robin function are defined by
\begin{equation}\label{eq:green}
 H(x,y):=\frac1{2\pi}\log\frac1{|x-y|}-G(x,y),\qquad R(x):=H(x,x).
\end{equation}
In particular, our \(H\) is the negative of the regular part used in \cite{MMT}. Both \(G\) and \(H\) are symmetric. The Green representation formula is
\(u(x)=\int_\Om G(x,y)\rho(y)\dd y\).
The maximum principle and interior estimates for harmonic functions give the following standard bounds; see \cite[Chapter~2]{GT}. For every \(K\Subset\Om\),
\begin{equation}\label{eq:greenbounds}
 \begin{split}
 &0<G(x,y)\le\frac1{2\pi}\log\frac{L_0}{|x-y|},\qquad
 \sup_{y\in\Om}\|G(\cdot,y)\|_{L^q(\Om)}\le C_{q,\Om},\\
 &|\nabla_xG(x,y)|\le\frac{C_K}{|x-y|},\qquad
 |H(x,y)|+|\nabla_xH(x,y)|\le C_K\quad(x\in K, y\in\Om).
 \end{split}
\end{equation}
Here \(L_0>\operatorname{diam}\Om\) is fixed, and pointwise estimates for singular kernels are understood for \(x\ne y\). Since \(\Omega\subset B_{L_0}(y)\), comparing \(G(\cdot,y)\) with the Green function of the disk with pole at $y$ yields the first inequality. Integrating in polar coordinates and setting \(t=\log(L_0/r)\), we obtain
\[
 \begin{aligned}
  \int_\Om G(x,y)^q\dd x
  &\le(2\pi)^{1-q}\int_0^{L_0}r\left(\log\frac{L_0}{r}\right)^q\dd r\\
  &=(2\pi)^{1-q}L_0^2\int_0^\infty t^qe^{-2t}\dd t<\infty.
 \end{aligned}
\]
The right-hand side is independent of \(y\), proving the \(L^q\) estimate.

Choose an open set \(U\) with \(K\Subset U\Subset\Om\). For each \(z\in U\), symmetry shows that \(H(z,\cdot)\) is harmonic in \(\Om\), with boundary values \((2\pi)^{-1}\log(1/|z-\xi|)\), \(\xi\in\partial\Om\). These values are uniformly bounded because \(\dist(U,\partial\Om)>0\) and \(\Om\) is bounded. The maximum principle therefore gives
\[
 \sup_{z\in U,\,y\in\Om}|H(z,y)|\le C_U.
\]
Applying the interior gradient estimate to \(H(\cdot,y)\) on \(U\), we find
\[
 \sup_{x\in K}|\nabla_xH(x,y)|
 \le\frac{C}{\dist(K,\partial U)}\sup_{z\in U}|H(z,y)|\le C_K.
\]
Both constants are independent of \(y\in\Om\). Differentiating \eqref{eq:green} now yields
\[
 |\nabla_xG(x,y)|\le\frac{1}{2\pi|x-y|}+C_K
 \le\frac{C'_K}{|x-y|},
\]
where the last inequality uses boundedness of \(\Om\).

For a sequence of solutions, {we add the subscript} \(j\) {to these quantities}; in particular, \(E_j=E(u_j)\). Define
\begin{equation}\label{eq:normalizations}
 w_j=\frac{u_j}{M_j},\qquad s_j=M_ju_j=M_j^2w_j,
 \qquad \dd\nu_j=\frac{\rho_j}{M_j}\dd x,
 \qquad \dd\Pi_j=\frac{P_j}{M_j^2}\dd x.
\end{equation}
Here \(w_j,s_j\) are normalized functions, and \(\nu_j,\Pi_j\) are finite nonnegative measures with \(\nu_j(\Om)=1\). The Green representation formula, \eqref{eq:energyidentity}, and the density formulas above give
\begin{equation}\label{eq:ratios}
 w_j=G\nu_j,\qquad E_j=\int_\Om s_j\dd\nu_j,
 \qquad \frac{\dd\Pi_j}{\dd\nu_j}=
 \frac{1-e^{-u_j^2}}{s_j}.
\end{equation}
Here \(\dd\Pi_j/\dd\nu_j\) is the Radon--Nikodym derivative. Since \(u_j>0\), the ratio is defined everywhere in \(\Om\).

The inequality \(1-e^{-t^2}\le\min\{t^2,1\}\) for \(t>0\) gives
\begin{equation}\label{eq:pressure-comparison}
 0\le P\le\min\{u,u^{-1}\}\rho,
 \qquad
 \ind_{\{s_j\ge L\}}\dd\Pi_j\le L^{-1}\dd\nu_j\quad(L>0).
\end{equation}

We regard \(\nu_j,\Pi_j\) as measures on \(\cl\Om\), assigning them zero mass on \(\partial\Om\). Weak convergence means convergence of integrals for every function in \(C(\cl\Om)\). The boundary mass of their weak limits is determined separately in Lemma~\ref{lem:boundary-zero}. Local weak convergence of rescaled measures on \(\R^2\) means convergence for every compactly supported continuous function.

For a Borel set \(B\), write \(\eta|_B(C)=\eta(B\cap C)\) for the restriction of a measure and \(\ind_B\) for the characteristic function. If \(T\) is measurable, its pushforward is defined by \((T_\#\eta)(B)=\eta(T^{-1}(B))\). For a finite signed measure \(\eta\), let \(\|\eta\|_{\TV(B)}=|\eta|(B)\) denote its total variation on \(B\). If \(\eta=f\lambda\), where \(\lambda\) is nonnegative, this equals \(\int_B|f|\dd\lambda\). When using weak convergence to pass to the limit in the mass of a ball, we require its boundary to have zero limiting measure, that is, \(\eta(\partial B_r(a))=0\); when two limiting measures occur, this condition is imposed for both. For \(t\in\R\), set \(t_+:=\max\{t,0\}\), and for \(t>0\), set \(\log^+t:=\max\{\log t,0\}\).

Constants denoted by the same letter may change from one formula to the next; subscripts indicate their parameter dependence. Unless stated otherwise, \(o(1)\) denotes a quantity tending to zero as the sequence index tends to infinity, with all auxiliary parameters fixed.

\subsection{Auxiliary results}
We recall the boundary monotonicity, blow-up, and Liouville classification results used below.

\begin{tool}[Proposition 2.1 in \cite{MMT}]\label{tool:boundary}
Let \(\Om\) be a smooth bounded planar domain. For \(q\in\partial\Om\), let \(n(q)\) be the outward unit normal and define the boundary normal coordinates
\[
 \Psi(q,t)=q-tn(q),\qquad q\in\partial\Om,\quad 0\le t\le\delta_0.
\]
Here \(t\) is the inward normal distance from \(q\). There exists a sufficiently small \(\delta_0>0\), depending only on \(\Om\), such that every point in the strip \(\{x\in\cl\Om:\dist(x,\partial\Om)\le\delta_0\}\) has a unique representation \(x=\Psi(q,t)\), where \(q\) is the nearest boundary point and
\(\dist(\Psi(q,t),\partial\Om)=t\). Let \(f\in C^1([0,\infty))\) be nonnegative and nondecreasing. Every positive solution of \(-\Delta v=f(v)\) with zero Dirichlet boundary data has the following property: for each fixed \(q\in\partial\Om\), the function
\[
 t\longmapsto v(\Psi(q,t))=v(q-tn(q))
\]
is strictly increasing on \([0,\delta_0]\). Moreover, there is a compact set \(K_{\mathrm{peak}}\Subset\Om\), depending only on \(\Om\), that contains every critical point, namely every point where \(\nabla v=0\), of every such solution.
\end{tool}
 For \(f(\xi)=\theta\xi e^{\xi^2}\), \(\xi\ge0\), we have \(f'(\xi)=\theta(1+2\xi^2)e^{\xi^2}>0\), and the constants above can be taken to be independent of \(\theta\).

\begin{tool}[]\label{tool:DTglobal}
Suppose that \(u_j\) satisfies \eqref{eq:pde}, \(\sup_j E_j<\infty\), and \(\gamma_j\to\infty\). Then, along a subsequence,
\(\theta_j\to0\) and \(\sqrt{\theta_j}\gamma_j\to c\in(0,\infty)\).
\end{tool}
This follows from \cite[Theorem 1.2(d)]{DT} with \(f_j\equiv1\); the constant coefficient satisfies its hypotheses. The first concentration point selected there is a global maximum point \cite[proof of Proposition 2.1, Step 1]{DT}, so the conclusion applies to \(\gamma_j\). We use this result only in Lemma~\ref{lem:logtheta}, where the contradiction assumption and the total mass estimate ensure bounded energy.

\begin{tool}[]\label{tool:DTlocal}
Let \(v_j>0\) solve \(-\Delta v_j=\theta_jv_je^{v_j^2}\) on \(B_{d_j}(0)\). Assume that \(\theta_j,d_j\) are bounded, \(\beta_j=v_j(0)\to\infty\), and \(\nabla v_j(0)=0\). Set
\begin{equation}\label{eq:Uz-mu_j}
    	U(z)=-\log(1+\lvert z\rvert^2/4),\quad\mu_j=(\theta_j\beta_j^2e^{\beta_j^2})^{-1/2}.
    \end{equation}
Assume that
\[
 \mu_j\to0,\qquad d_j/\mu_j\to\infty,
 \qquad \beta_j(v_j(\mu_j z)-\beta_j)\to U(z)
 \quad\hbox{in } C^2_{\mathrm{loc}}(\R^2),
\]
and that

\begin{equation}\label{eq:DTconditions}
 |x|^2\theta_jv_j^2e^{v_j^2}\le C,
 \qquad |x||\nabla v_j|\le C/\beta_j\quad\hbox{in } B_{d_j}.
\end{equation}
Let \(B_j=B_j(r)\) be the radial solution of the same equation, that is,
\[
 -B_j''(r)-\frac1rB_j'(r)=\theta_jB_j(r)e^{B_j(r)^2}
 \quad(0<r<d_j),\qquad B_j(0)=\beta_j,\quad B_j'(0)=0.
\]
In comparisons on balls, \(B_j(x)\) means \(B_j(|x|)\). Then
\begin{equation}\label{eq:DTcomparison}
 \sup_{0\le r\le d_j}|\cl v_j(r)-B_j(r)|=o(\beta_j^{-1}),
 \quad |v_j-B_j|\le C\beta_j^{-1},
 \quad |\nabla(v_j-B_j)|\le\frac{C}{\beta_jd_j},
\end{equation}
where
\[
 \cl v_j(r)=\frac1{2\pi}\int_0^{2\pi}v_j(re^{i\alpha})\dd\alpha
 =\frac1{2\pi r}\int_{\partial B_r}v_j\dd\sigma\quad(r>0),
 \qquad \cl v_j(0)=v_j(0).
\]
Set \(t_j(r)=\log(1+r^2/(4\mu_j^2))\). The radial solution has the expansion
\begin{equation}\label{eq:radialexpansion}
 B_j(r)=\beta_j-\frac{t_j(r)}{\beta_j}-\frac{t_j(r)}{\beta_j^3}
 +O(\beta_j^{-2})\qquad(0\le t_j(r)\le\beta_j^2).
\end{equation}
\end{tool}
For constant coefficients, the local comparison is given in \cite[Proposition~3.1(a), (b)]{DT}. The radial solution and logarithmic variable are defined in (3.8) and (3.11), and the expansion is (3.9). The hypotheses are listed in (3.1)--(3.7). We verify \eqref{eq:DTconditions} in Lemma~\ref{lem:budget}; the weighted integral bounds needed in each application are given in Lemmas~\ref{lem:maxbubble}, \ref{lem:lowdisk}, and Proposition~\ref{prop:neighborhood}.

\begin{tool}[Chen--Li classification theorem \cite{CL}]\label{tool:chenli}
If \(V\in C^2(\R^2)\), \(-\Delta V=e^{2V}\), and \(\int_{\R^2}e^{2V}<\infty\), then
\[
 V(z)=\log\frac{2\lambda}{1+\lambda^2|z-z_0|^2}
 \quad(\lambda>0, z_0\in\R^2),
 \qquad \int_{\R^2}e^{2V}=4\pi.
\]
\end{tool}
\section{Conformal change of variables and the total mass estimate}\label{sec:mass}

\begin{revision}
\subsection{The weighted Pohozaev identity and total mass estimates}
Fix a conformal diffeomorphism \(\phi:\D\to\Om\), identifying planar coordinates with complex numbers. Set
    \begin{equation}\label{eq:v-a-h-g-G}
    	v=u\circ\phi,\quad a=|\phi'|^2=e^{2h},\quad
    	h=\log|\phi'|,\quad g(z)=1+z\cdot\nabla h(z),\quad G_* =\max_{\cl\D}g.
    \end{equation}
Here \(\phi'\) is the complex derivative, and \(z\cdot\nabla h\) is the Euclidean inner product with \(z\) viewed as a real two-dimensional vector. The function \(h\) is smooth and harmonic; since \(g(0)=1\), we have \(G_*\ge1\). The conformal map has Jacobian \(a=|\phi'|^2\), and \(\Delta v=a\,\bigl((\Delta u)\circ\phi\bigr)\). Thus \(v=0\) on \(\partial\D\), and the transformed equation and densities are
\[
 -\Delta v=\rho_\D:=\theta av e^{v^2},\qquad
 P_\D=\theta a(e^{v^2}-1),\qquad
 \int_\D\rho_\D=M.
\]
Thus \(\rho_\D\) and \(P_\D\) are the densities of \(\rho\dd x\) and \(P\dd x\) in disk coordinates. Dirichlet energy is invariant under this conformal change of variables.

We first state the Pohozaev identity for this weighted equation. We then expand the negative outward normal derivative \(-\partial_n v\) on the unit circle in a Fourier series and use Parseval's identity to express the boundary integral.

\begin{lemma}\label{lem:moment}
The pullback \(v\) satisfies the Pohozaev identity
\begin{equation}\label{eq:weighted-pohozaev}
 \begin{aligned}
 \frac12\int_{\partial\D}(\partial_n v)^2\dd\sigma
 &=\theta\int_\D\left(a+\frac12z\cdot\nabla a\right)(e^{v^2}-1)\dd z\\
 &=\int_\D gP_\D\dd z.
 \end{aligned}
\end{equation}
Expanding \(-\partial_n v\) in a Fourier series and applying Parseval's identity gives the equivalent form
\begin{equation}\label{eq:moment}
 M^2+2\sum_{n=1}^{\infty}\left|\int_\D z^n\rho_\D(z)\dd z\right|^2
 =4\pi\int_\D gP_\D\dd z.
\end{equation}
\end{lemma}

\begin{proof}
Apply \cite[Theorem~2.2]{Battaglia2019} with \(G(z,s)=\frac{\theta}{2}a(z)(e^{s^2}-1)\). Then \(\partial_sG(z,v)=\rho_\D\), and \(G(z,v)=0\) on \(\partial\D\) because \(v=0\) there. Together with \(a+\frac12z\cdot\nabla a=ag\), this gives \eqref{eq:weighted-pohozaev}.

We next evaluate the boundary integral by Fourier expansion. Write \(b(t)=-\partial_n v(e^{it})\), where \(n\) is the outward unit normal to \(\partial\D\), and set \(c_n=\int_\D z^n\rho_\D(z)\dd z\).
The Green representation formula on the disk gives
\[
 b(t)=\frac1{2\pi}\int_\D K(z,t)\rho_\D(z)\dd z,
 \qquad K(z,t)=\frac{1-|z|^2}{|e^{it}-z|^2}.
\]
The Poisson kernel has the expansion
\[
 K(z,t)=1+\sum_{n\ge1}(z^ne^{-int}+\cl z^{\,n}e^{int}).
\]
Since \(K\ge0\) and \(\int_0^{2\pi}K(z,t)\dd t=2\pi\), we have
\(\int_\D\int_0^{2\pi}K\rho_\D\dd t\dd z=2\pi M<\infty\). Fubini's theorem and orthogonality yield
\[
 \int_0^{2\pi}b(t)\dd t=M,
 \qquad \int_0^{2\pi}b(t)e^{int}\dd t=c_n\quad(n\ge1).
\]
The function \(b\) is smooth and real-valued, so its positive and negative frequency coefficients are complex conjugates. Parseval's identity therefore gives
\[
 \int_0^{2\pi}b(t)^2\dd t
 =2\pi\left[\left|\frac M{2\pi}\right|^2
       +2\sum_{n\ge1}\left|\frac{c_n}{2\pi}\right|^2\right]
 =\frac1{2\pi}\left(M^2+2\sum_{n\ge1}|c_n|^2\right).
\]

Substituting Parseval's identity into \eqref{eq:weighted-pohozaev}, we obtain
\[
 \int_\D gP_\D\dd z
 =\frac12\int_0^{2\pi}b(t)^2\dd t
 =\frac1{4\pi}\left(M^2+2\sum_{n\ge1}|c_n|^2\right),
\]
which is \eqref{eq:moment}.
\end{proof}
\end{revision}

The preceding identity gives a uniform bound for \(M\) and a decay estimate as \(\theta\to0\).
\begin{lemma}\label{lem:mass}
Let \(\lambda_1(\Om)\) be the first Dirichlet eigenvalue of \(-\Delta\) on \(\Om\). There exists a constant \(C_\Om\), depending only on \(\Om\), such that
\begin{equation}\label{eq:massbounds}
 0<\theta<\lambda_1(\Om),\qquad M\le C_\Om,
 \qquad M\le\frac{C_\Om}{\sqrt{\log(1/\theta)}}
 \quad(0<\theta<e^{-2}).
\end{equation}
Consequently, if \(\theta_j\to0\), then \(M_j\to0\) and \(u_j\to0\) in \(L^q(\Om)\) for every \(1\le q<\infty\).
\end{lemma}
\begin{proof}
Testing the equation {with} a positive first Dirichlet eigenfunction \(\varphi_1\), we obtain
\[
 \theta\int_\Om ue^{u^2}\varphi_1
 =\lambda_1(\Om)\int_\Om u\varphi_1
 >\theta\int_\Om u\varphi_1,
\]
Thus \(\theta<\lambda_1(\Om)\). By \eqref{eq:pressure-comparison},
\begin{equation}\label{eq:int_Omega-P-dx}
 \int_\Om P\dd x\le\int_\Om\rho\dd x=M.
\end{equation}
Set \(C_0=4\pi G_*\). The integral identity \eqref{eq:moment} gives
\[
 M^2\le4\pi\int_\D gP_\D
 \le C_0\int_\Om P\le C_0M,
\]
and therefore \(M\le C_0\).

For any \(T>1\), use \(P\le\theta e^{T^2}\) on \(\{u\le T\}\) and \eqref{eq:pressure-comparison} on \(\{u>T\}\) to obtain
\[
 M^2\le C_0\left(|\Om|\theta e^{T^2}+\frac MT\right)
 \le C_0|\Om|\theta e^{T^2}+\frac12M^2+\frac{C_0^2}{2T^2}.
\]
For \(0<\theta<e^{-2}\), choose \(T=\sqrt{\log(1/\theta)/2}\). Then
\[
 M^2\le\frac{2C_0^2}{\log(1/\theta)}+2C_0|\Om|\sqrt\theta
 \le\frac{C_\Om}{\log(1/\theta)}.
\]
This proves \eqref{eq:massbounds}. If \(\theta_j\to0\), then \(M_j\to0\). For every finite \(q\), Minkowski's inequality and \eqref{eq:greenbounds} give
\[
 \|u_j\|_{L^q(\Om)}
 \le\int_\Om\|G(\cdot,y)\|_{L^q(\Om)}\rho_j(y)\dd y
 \le C_{q,\Om}M_j\to0.
\]
This completes the proof.
\end{proof}

\subsection{Estimates in a boundary strip}
\begin{lemma}\label{lem:strip}
Let \(u_j\) be positive solutions of \eqref{eq:pde}, with \(\nu_j,\Pi_j\) defined by \eqref{eq:normalizations}, and write
\(d(x)=\dist(x,\partial\Om)\). There exist constants \(\delta_0>0\) and \(C_\Om>0\), depending only on \(\Om\), such that, for every \(j\) and \(0<\eps<\delta_0\),
\begin{equation}\label{eq:strip}
 \nu_j\{d<\eps\}\le C_\Om\eps,
 \qquad \Pi_j\{d<\eps\}\le C_\Om\eps\Pi_j(\Om).
\end{equation}
\end{lemma}
\begin{proof}
In the normal coordinates of Theorem~\ref{tool:boundary}, decrease \(\delta_0\) so that the area element satisfies
\[
 \dd x=J(q,t)\dd\sigma(q)\dd t,
 \qquad \tfrac12\le J(q,t)\le\tfrac32\quad(0\le t\le\delta_0).
\]
Let \(F\ge0\) be nondecreasing along inward normals, and set \(S_\eps=\{d<\eps\}\). Then
\[
 \int_0^\eps F(\Psi(q,t))J(q,t)\dd t
 \le\tfrac32\eps F(\Psi(q,\eps)),\qquad
 \int_\eps^{\delta_0}F(\Psi(q,t))J(q,t)\dd t
 \ge\tfrac12(\delta_0-\eps)F(\Psi(q,\eps)).
\]
Integrating in \(q\) and writing \(I_\eps=\int_{S_\eps}F\), \(T=\int_\Om F\), we obtain
\[
 (\delta_0-\eps)I_\eps\le3\eps(T-I_\eps),
 \qquad I_\eps\le\frac{3\eps}{\delta_0+2\eps}T\le\frac{3\eps}{\delta_0}T.
\]
Boundary monotonicity shows that \(\rho_j\) and \(P_j\), defined in \eqref{eq:notation}, satisfy these conditions. Applying the estimate to each and dividing by \(M_j\) and \(M_j^2\), respectively, proves \eqref{eq:strip}.
\end{proof}

\subsection{Exponential integrability and truncation}
\begin{lemma}\label{lem:jensen}
For every \(0<\alpha<4\pi\),
\begin{equation}\label{eq:globaljensen}
 \int_\Om e^{\alpha w_j}\dd x\le C_{\alpha,\Om},
\end{equation}
where \(w_j\) is defined in \eqref{eq:normalizations}. Suppose that \(\nu_j\weak\nu\), \(B_{2r}(x_0)\Subset\Om\), and \(\bar p>\nu(\cl B_{2r}(x_0))\). For fixed \(0<\alpha<4\pi\), we then have
\begin{equation}\label{eq:localjensen}
 \int_{B_r(x_0)}e^{(\alpha/\bar p)w_j}\dd x\le C_{r,\bar p,\alpha,\Omega}
\end{equation}
for all sufficiently large \(j\). If \(0<L<\alpha/\bar p\), then
\begin{equation}\label{eq:lowheight}
 \nu_j(B_r(x_0)\cap\{s_j\le L\})+
 \Pi_j(B_r(x_0)\cap\{s_j\le L\})\le C_{r,L,\bar p,\alpha,\Omega}\theta_j.
\end{equation}
\end{lemma}
\begin{proof}
Since \(w_j=G\nu_j\) and \(\nu_j(\Om)=1\), Jensen's inequality and \eqref{eq:greenbounds} give
\[
 e^{\alpha w_j(x)}
 \le\int_\Om e^{\alpha G(x,y)}\dd\nu_j(y)
 \le\int_\Om\left(\frac{L_0}{|x-y|}\right)^{\alpha/(2\pi)}\dd\nu_j(y).
\]
The singularity can be handled by first truncating \(G\) and then applying monotone convergence. Since \(\alpha/(2\pi)<2\), integration in polar coordinates gives
\[
 K_{\alpha,\Om}:=\sup_{y\in\Om}
 \int_\Om\left(\frac{L_0}{|x-y|}\right)^{\alpha/(2\pi)}\dd x<\infty.
\]
Integrating in \(x\) and applying Tonelli's theorem proves \eqref{eq:globaljensen}.

For the local estimate, abbreviate \(B_r=B_r(x_0)\) and set \(m_j=\nu_j(B_{2r})\). The weak convergence inequality for closed sets gives
\[
 \limsup_jm_j\le\nu(\cl B_{2r})<\bar p.
\]
Split the integral defining \(w_j=G\nu_j\) as \(w_j=a_j+b_j\), where
\[
 a_j(x)=\int_{B_{2r}}G(x,y)\dd\nu_j(y),\qquad
 b_j(x)=\int_{\Om\setminus B_{2r}}G(x,y)\dd\nu_j(y).
\]
For \(x\in B_r\), the second term satisfies
\(0\le b_j(x)\le(2\pi)^{-1}\log(L_0/r)\). For all sufficiently large \(j\), the measure
\(\nu_j|_{B_{2r}}/\bar p\) has mass less than one. {Add a point to the domain of integration}, assign it mass \(1-m_j/\bar p\), and define the kernel to be zero there. Jensen's inequality for the resulting probability measure yields
\[
 e^{(\alpha/\bar p)a_j(x)}
 \le1-\frac{m_j}{\bar p}
 +\frac1{\bar p}\int_{B_{2r}}e^{\alpha G(x,y)}\dd\nu_j(y).
\]
The integral of the right-hand side over \(B_r\) is at most \(|B_r|+K_{\alpha,\Om}\). Together with the bound for \(b_j\), this proves \eqref{eq:localjensen}.

We now estimate the densities on \(\{s_j\le L\}\). Suppose that a measurable set \(D\subset\Om\) satisfies
\(\int_D e^{\Lambda w_j}\le C\) for a fixed \(\Lambda>0\). On \(\{s_j\le L\}\), the identities \(u_j^2=s_jw_j\) and \(e^t-1\le te^t\) give
\[
 \frac{\dd\nu_j}{\dd x}=\theta_jw_je^{u_j^2}\le\theta_jw_je^{Lw_j},
 \qquad
 \frac{\dd\Pi_j}{\dd x}\le\theta_jw_j^2e^{u_j^2}\le\theta_jw_j^2e^{Lw_j}.
\]
If \(q\ge1\) and \(qL<\Lambda\), then
\((t^q+t^{2q})e^{qLt}\le C_{q,L,\Lambda}e^{\Lambda t}\), hence
\begin{equation}\label{eq:lowheight-Lq}
 \left\|\ind_{\{s_j\le L\}}\frac{\dd\nu_j}{\dd x}\right\|_{L^q(D)}
 +\left\|\ind_{\{s_j\le L\}}\frac{\dd\Pi_j}{\dd x}\right\|_{L^q(D)}
 \le C_{D,q,L,\Lambda}\theta_j.
\end{equation}
Taking \(D=B_r(x_0)\), \(\Lambda=\alpha/\bar p\), and \(q=1\) proves \eqref{eq:lowheight}.
\end{proof}

\begin{remark}\label{rem:global-lowheight}
For \(D=\Om\), \eqref{eq:globaljensen} shows that \eqref{eq:lowheight-Lq} holds whenever \(qL<4\pi\). In particular, for \(0<L<4\pi\),
\[
 \nu_j(\{s_j\le L\})+\Pi_j(\{s_j\le L\})\le C_{L,\Om}\theta_j.
\]
In either the global or local setting above, combining \eqref{eq:lowheight-Lq} with \(q=1\) and \eqref{eq:pressure-comparison} gives, for every bounded nonnegative function \(\psi\),
\begin{equation}\label{eq:pressure-split}
 \int_D\psi\dd\Pi_j
 \le C_{D,L}\|\psi\|_\infty\theta_j
   +\frac1L\int_D\psi\dd\nu_j.
\end{equation}
The global estimate requires \(L<4\pi\), and the local estimate requires \(L<4\pi/\bar p\). The constants also depend on the gap between the chosen exponent and its critical value.
\end{remark}

\section{Local radial comparison and estimates at a maximum}\label{sec:local}
This section bounds the first factor \(\sqrt{\theta_j}\gamma_j\) in the energy estimate and establishes the logarithmic energy bound used later. We first derive the pointwise and gradient bounds required for radial comparison from a local weighted integral condition. Combining the radial expansion with the Green representation formula then gives the asymptotic relations in Lemma~\ref{lem:matching}. We verify the hypotheses of these local results at a global maximum to prove Proposition~\ref{prop:maxbound}. The same results will be applied to the selected local maxima in Sections~\ref{sec:center} and \ref{sec:absorb}, where their hypotheses will be checked.

\subsection{Local compactness and the hypotheses of radial comparison}
Let \(\zeta_j\) be \(C^2\) functions on domains exhausting \(\R^2\), with \(\zeta_j(0)=0\) and \(\zeta_j\le o(1)\) uniformly on each fixed ball. Fix \(R>0\), and suppose that \(0\le-\Delta\zeta_j\le C_R\) on \(B_{2R}\), with \(C_R\) independent of \(j\). Choose \(\eps_j\downarrow0\), possibly depending on \(R\), such that \(\zeta_j\le\eps_j\) on that ball, and set \(v_j=\eps_j-\zeta_j\ge0\). Its circular mean satisfies
\[
 \cl v_j(r)=v_j(0)+\frac1{2\pi}\int_0^r\frac1t\int_{B_t}\Delta v_j\dd x\dd t
 \le\eps_j+\frac{C_Rr^2}{4}.
\]
Multiplying by \(2\pi r\) and integrating from \(0\) to \(2R\), we obtain
\[
 \int_{B_{2R}}v_j\dd x\le C_R(1+\eps_j).
\]
Since \(\Delta v_j=-\Delta\zeta_j\ge0\), the function \(v_j\) is nonnegative and subharmonic. For each \(z\in B_R\),
the {mean-value inequality} on \(B_R(z)\subset B_{2R}\) gives
\[
 0\le v_j(z)\le\frac1{\pi R^2}\int_{B_R(z)}v_j
 \le C_R(1+\eps_j).
\]
Since \(\zeta_j=\eps_j-v_j\), this gives uniform upper and lower bounds for \(\zeta_j\). Interior estimates on slightly smaller balls then yield, for every finite \(q>1\),
the bound
\[
 \|\zeta_j\|_{W^{2,q}(B_{R/2})}
 \le C_{R,q}\bigl(\|\zeta_j\|_{L^q(B_R)}+\|\Delta\zeta_j\|_{L^q(B_R)}\bigr)\le C_{R,q}.
\]
In the rescalings below, the right-hand sides are smooth functions of \(\zeta_j\), with derivatives uniformly bounded on bounded ranges. Sobolev embedding and Schauder estimates therefore give subsequential convergence in \(C^2_{\mathrm{loc}}\).

\begin{lemma}\label{lem:budget}
Let \(x_j\in K\Subset\Om\), \(\beta_j=u_j(x_j)\to\infty\), and \(\nabla u_j(x_j)=0\). Assume that
\begin{equation}\label{eq:localassumptions}
 \log\theta_j=o(\beta_j^2),\qquad
 \beta_j(u_j(x_j+\mu_j z)-\beta_j)\to U(z)
 \quad\hbox{in }C^2_{\mathrm{loc}}(\R^2),
\end{equation}
where \(U\) and \(\mu_j\) are defined in \eqref{eq:Uz-mu_j}.
Suppose that \(\sup_jM_j<\infty\), \(\sup_j\theta_j<\infty\), and
\begin{equation}\label{eq:budget}
 \int_{B_{L_j}(x_j)}u_j\rho_j\dd x\longrightarrow4\pi,
 \qquad L_j=e^{-\beta_j^2/32}.
\end{equation}
Then Theorem~\ref{tool:DTlocal} applies with comparison radius
\(d_j=e^{-\beta_j^2/16}\).
\end{lemma}
\begin{proof}
The relation \(\log\theta_j=o(\beta_j^2)\) and the definition of \(\mu_j\) give
\(\log(1/\mu_j)=\beta_j^2/2+o(\beta_j^2)\). Hence \(\mu_j/d_j\to0\) and \(d_j/L_j\to0\). It remains to prove the two bounds in \eqref{eq:DTconditions}.

Fix \(T>0\), and set \(m_T=\int_{B_T}e^{2U}>0\). Changing variables by \(x=x_j+\mu_jz\) and using the assumed \(C^2_{\mathrm{loc}}\) convergence, we obtain
\(\beta_j\int_{B_{T\mu_j}(x_j)}\rho_j\to m_T\).
If \(z\in B_{L_j}(x_j)\) and \(y\in B_{T\mu_j}(x_j)\), then \(|z-y|\le2L_j\) for all sufficiently large \(j\). By \eqref{eq:green},
\[
 G(z,y)\ge\frac1{2\pi}\log\frac1{2L_j}-C_K
 =\frac{\beta_j^2}{64\pi}-O_K(1).
\]
Since \(G\ge0\) and \(\rho_j\ge0\), retaining only the integral over \(B_{T\mu_j}(x_j)\) in the Green representation formula gives
\[
 u_j(z)\ge\left(\frac{\beta_j^2}{64\pi}-C\right)
 \frac{m_T+o(1)}{\beta_j}\ge c\beta_j.
\]
Using \eqref{eq:budget}, we obtain
\begin{equation}\label{eq:budgetlower}
 u_j\ge c\beta_j\quad\hbox{in }B_{L_j}(x_j),\qquad
 \int_{B_{L_j}(x_j)}\rho_j\dd x
 \le\frac1{c\beta_j}\int_{B_{L_j}(x_j)}u_j\rho_j\dd x
 \le\frac C{\beta_j}.
\end{equation}

Set \(V_j=\theta_ju_j^2e^{u_j^2}=u_j\rho_j\), and define on \(\cl B_{L_j}(x_j)\)
\(D_j(z)=\min\{|z-x_j|,L_j-|z-x_j|\}\).
We first prove that \(\sup_{\cl B_{L_j}(x_j)}D_j^2V_j\le C\). Otherwise, pass to a subsequence and choose \(y_j\) where the continuous function \(D_j^2V_j\) attains its maximum on the closed ball, with
\[
 D_j(y_j)^2V_j(y_j)
 =\max_{y\in\cl B_{L_j}(x_j)}D_j(y)^2V_j(y)\to\infty.
\]
Set \(b_j=u_j(y_j)\) and \(\widehat\mu_j=V_j(y_j)^{-1/2}\). Since \(D_j(y_j)\le L_j\to0\), we have \(V_j(y_j)\to\infty\). Boundedness of \(\theta_j\) and the identity \(V_j(y_j)=\theta_jb_j^2e^{b_j^2}\) then give \(b_j\to\infty\). Also,
\[
 \frac{\widehat\mu_j}{D_j(y_j)}
 =\frac1{\sqrt{D_j(y_j)^2V_j(y_j)}}\to0.
\]
Fix \(T>0\). For all sufficiently large \(j\), \(T\widehat\mu_j<D_j(y_j)\le L_j-|y_j-x_j|\), so \(y_j+\widehat\mu_jz\in B_{L_j}(x_j)\) whenever \(|z|\le T\). The definition of \(D_j\) gives \(|D_j(y)-D_j(y')|\le|y-y'|\); hence
\[
 D_j(y_j+\widehat\mu_jz)\ge D_j(y_j)-T\widehat\mu_j>0.
\]
Substituting \(y_j+\widehat\mu_jz\) into \(D_j(y)^2V_j(y)\le D_j(y_j)^2V_j(y_j)\), we obtain
\[
 \frac{V_j(y_j+\widehat\mu_j z)}{V_j(y_j)}
 \le\left(\frac{D_j(y_j)}{D_j(y_j+\widehat\mu_jz)}\right)^2
 \le\left(1-\frac{T\widehat\mu_j}{D_j(y_j)}\right)^{-2}
 =1+o(1).
\]
The final error is independent of \(z\), so the estimate is uniform on each fixed ball. If \(u_j(y_j+\widehat\mu_j z)>b_j\), taking logarithms gives \(u_j^2-b_j^2\le o(1)\). If \(u_j(y_j+\widehat\mu_j z)\le b_j\), then \(u_j^2-b_j^2\le0\). Therefore
\[
 \zeta_j(z):=b_j(u_j(y_j+\widehat\mu_j z)-b_j)\le o(1),\qquad
 \zeta_j(0)=0,\qquad
 0\le-\Delta\zeta_j=\frac{u_j}{b_j}e^{u_j^2-b_j^2}\le1+o(1).
\]
Here \(\Delta\) acts on \(z\), and every occurrence of \(u_j\) on the right is evaluated at \(y_j+\widehat\mu_jz\). The local \(W^{2,q}\) and Schauder estimates at the beginning of this subsection give a subsequential limit \(\zeta\) satisfying \(-\Delta\zeta=e^{2\zeta}\) and \(\zeta\le0=\zeta(0)\). For fixed \(T>0\), make the change of variables \(x=y_j+\widehat\mu_jz\) on \(B_{T\widehat\mu_j}(y_j)\). Since \(\widehat\mu_j^2V_j(y_j+\widehat\mu_jz)\to e^{2\zeta(z)}\), this ball lies in \(B_{L_j}(x_j)\), and \eqref{eq:budget} holds, Fatou's lemma gives \(\int_{B_T}e^{2\zeta}\le4\pi\). Letting \(T\to\infty\) and applying monotone convergence yields \(\int_{\R^2}e^{2\zeta}\le4\pi\). Theorem~\ref{tool:chenli} therefore implies \(\zeta=U\) and \(\int_{\R^2}e^{2U}=4\pi\).
The distance between the two centers tends to infinity relative to either rescaling parameter. First, \(D_j(y_j)\le|y_j-x_j|\) gives \(|y_j-x_j|/\widehat\mu_j\to\infty\). If \(|y_j-x_j|/\mu_j\) were bounded along a subsequence, \eqref{eq:localassumptions} and \eqref{eq:Uz-mu_j} would give \(\mu_j^2V_j(y_j)=O(1)\). Together with \(D_j(y_j)\le|y_j-x_j|\le C\mu_j\), this would imply \(D_j(y_j)^2V_j(y_j)=O(1)\), contrary to the choice of \(y_j\).
Thus, for fixed \(T\), the balls \(B_{T\mu_j}(x_j)\) and \(B_{T\widehat\mu_j}(y_j)\) are eventually disjoint.
Since \(\mu_j/L_j\to0\) and \(T\widehat\mu_j<D_j(y_j)\le L_j-|y_j-x_j|\), both balls lie in \(B_{L_j}(x_j)\). The changes of variables \(x=x_j+\mu_jz\) and \(x=y_j+\widehat\mu_jz\), together with local convergence of the two rescaled functions, show that the integral of \(u_j\rho_j\) over each ball tends to \(m_T\). Disjointness and \(u_j\rho_j\ge0\) give
\[
 \liminf_j\int_{B_{L_j}(x_j)}u_j\rho_j\ge2m_T.
\]

Taking \(T\) large enough that \(2m_T>4\pi\) contradicts \eqref{eq:budget}. Hence \(\sup D_j^2V_j\le C\). Since \(D_j(z)=|z-x_j|\) on \(B_{L_j/2}(x_j)\), we obtain
\begin{equation}\label{eq:pointbudget}
 |z-x_j|^2\theta_ju_j(z)^2e^{u_j(z)^2}\le C.
\end{equation}

Because \(d_j/L_j\to0\), we have \(3d_j<L_j\) for all sufficiently large \(j\), and \eqref{eq:pointbudget} holds on \(B_{d_j}(x_j)\). To prove the gradient estimate, fix \(z\) with \(r=|z-x_j|\in(0,d_j]\). If \(y\in B_{r/2}(z)\), then \(r/2\le|y-x_j|\le3r/2\), and \(B_{r/2}(z)\subset B_{L_j/2}(x_j)\). Thus \eqref{eq:budgetlower} and \eqref{eq:pointbudget} imply \(\rho_j\le C/(\beta_jr^2)\) on this ball. {In the Green representation formula for the gradient, split the domain of integration into} \(B_{r/2}(z)\), \(B_{L_j}(x_j)\setminus B_{r/2}(z)\), and \(\Om\setminus B_{L_j}(x_j)\). By \eqref{eq:greenbounds}, the absolute values of these three contributions satisfy, respectively,
\[
 \frac{C}{\beta_jr^2}\int_{B_{r/2}(z)}\frac{\dd y}{|z-y|}
 \le\frac{C}{\beta_jr},\qquad
 \frac Cr\int_{B_{L_j}(x_j)}\rho_j\le\frac{C}{\beta_jr},\qquad
 \frac{CM_j}{L_j}.
\]
Since \(\sup_jM_j<\infty\),
\[
 \beta_jM_j\frac{d_j}{L_j}\le C\beta_je^{-\beta_j^2/32}\to0,
\]
the outer integral is also at most \(C/(\beta_jr)\). Together with \eqref{eq:pointbudget}, this proves \eqref{eq:DTconditions}. All remaining hypotheses of Theorem~\ref{tool:DTlocal} follow from the assumptions of the lemma.
\end{proof}

\subsection{Asymptotics of the source integral and the parameter}
\begin{lemma}\label{lem:matching}
Under the hypotheses of Lemma~\ref{lem:budget}, fix \(\tau\in(1/16,1/2)\) and set
\[
 r_j(\tau)=e^{-\tau\beta_j^2},\qquad
 Q_j(r)=\int_{B_r(x_j)}\rho_j\dd x,\qquad
 I_{\mathrm{out},j}(r)=\int_{\Om\setminus B_r(x_j)}G(x_j,y)\rho_j(y)\dd y.
\]
Then
\begin{align}
 Q_j(r_j(\tau))&=\frac{4\pi}{\beta_j}+\frac{4\pi}{\beta_j^3}+o(\beta_j^{-3}),\label{eq:flux}\\
 \log\frac{\theta_j\beta_j^2}{4}
 &=4\pi R(x_j)-1-\beta_j I_{\mathrm{out},j}(r_j(\tau))+o(1),
 \label{eq:matching}
\end{align}
where \(R\) is the Robin function defined in \eqref{eq:green}.
\end{lemma}
\begin{proof}
Set \(\ell_j=\log(\theta_j\beta_j^2/4)=o(\beta_j^2)\). At the comparison radius, the logarithmic variable in the radial expansion satisfies
\(t_j(d_j)=\tfrac78\beta_j^2+\ell_j+o(1)<\beta_j^2\) for all sufficiently large \(j\). On any fixed closed subinterval of \((1/16,1/2)\),
\begin{equation}\label{eq:tlog}
 t_j(r_j(\tau))=(1-2\tau)\beta_j^2+\ell_j+O(e^{-c\beta_j^2})
\end{equation}
uniformly, for some \(c>0\). Thus \eqref{eq:radialexpansion} applies at these radii.
Let \(m_j(r)=(2\pi)^{-1}\int_0^{2\pi}u_j(x_j+re^{i\alpha})\dd\alpha\) be the circular mean about \(x_j\). Integrating the equation over a ball gives
\(-2\pi r m_j'(r)=Q_j(r)\). Since \(\rho_j>0\), the function \(Q_j(r)=\int_{B_r(x_j)}\rho_j\) is nondecreasing. Choose fixed numbers \(1/16<\tau_-<\tau<\tau_+<1/2\), and set \(r_-=r_j(\tau_+)\), \(r=r_j(\tau)\), and \(r_+=r_j(\tau_-)\). Integrating \(-2\pi m_j'(t)=Q_j(t)/t\) over \([r_-,r]\) and \([r,r_+]\), and using monotonicity of \(Q_j\), we obtain
\begin{equation}\label{eq:secants}
 \frac{2\pi[m_j(r_-)-m_j(r)]}{\log(r/r_-)}
 \le Q_j(r)\le
 \frac{2\pi[m_j(r)-m_j(r_+)]}{\log(r_+/r)}.
\end{equation}
By \eqref{eq:DTcomparison}, \eqref{eq:radialexpansion}, and \eqref{eq:tlog}, at each of the three radii,
\[
 m_j(r)=\beta_j-
 \left(\frac1{\beta_j}+\frac1{\beta_j^3}\right)
 (\beta_j^2+\ell_j+2\log r)+o(\beta_j^{-1}).
\]
Upon subtraction, the constant terms arising from \(\beta_j^2+\ell_j\) cancel. For example,
\[
 m_j(r_-)-m_j(r)
 =\left(\frac2{\beta_j}+\frac2{\beta_j^3}\right)\log(r/r_-)
   +o(\beta_j^{-1}).
\]
Dividing by \(\log(r/r_-)=(\tau_+-\tau)\beta_j^2\) gives an error of precisely \(o(\beta_j^{-3})\).
The quotient on the right of \eqref{eq:secants} has the same expansion. Since \(Q_j(r)\) lies between the two quotients, it satisfies \eqref{eq:flux}.

The same expansion of the function gives
\begin{equation}\label{eq:meanexpansion}
 m_j(r_j(\tau))=2\tau\beta_j-
 \frac{\ell_j+1-2\tau}{\beta_j}+o(\beta_j^{-1});
\end{equation}
where \(\ell_j/\beta_j^3=o(\beta_j^{-1})\).
For each fixed \(y\), the circular mean of \(\log|x_j+re^{it}-y|\) is \(\log\max\{r,|x_j-y|\}\), while that of \(H(\cdot,y)\) is \(H(x_j,y)\). Substituting these two mean-value identities into the Green representation formula for \(u_j\) and interchanging the order of integration gives
\begin{equation}\label{eq:circleexact}
 m_j(r)=\frac{Q_j(r)}{2\pi}\log\frac1r
 -\int_{B_r(x_j)}H(x_j,y)\rho_j(y)\dd y
 +I_{\mathrm{out},j}(r).
\end{equation}
The integral involving the regular part equals \(R(x_j)Q_j(r)+O(rQ_j(r))=4\pi R(x_j)/\beta_j+o(\beta_j^{-1})\). By \eqref{eq:flux}, the first term is \(2\tau\beta_j+2\tau/\beta_j+o(\beta_j^{-1})\). Substituting into \eqref{eq:circleexact} and comparing with \eqref{eq:meanexpansion}, we obtain
\[
 2\tau\beta_j-\frac{\ell_j+1-2\tau}{\beta_j}
 =2\tau\beta_j+\frac{2\tau-4\pi R(x_j)}{\beta_j}
   +I_{\mathrm{out},j}(r_j(\tau))+o(\beta_j^{-1}).
\]
Cancelling \(2\tau\beta_j\), multiplying by \(\beta_j\), and rearranging proves \eqref{eq:matching}.
\end{proof}

\subsection{Estimates at a global maximum}
\begin{lemma}\label{lem:logtheta}
Every sequence of solutions with \(\gamma_j\to\infty\) satisfies
\(\log\theta_j=o(\gamma_j^2)\), where \(\gamma_j\) is defined in \eqref{eq:notation}.
\end{lemma}
\begin{proof}
Lemma~\ref{lem:mass} bounds the positive part of \(\log\theta_j\). If its negative part were not \(o(\gamma_j^2)\), there would be a subsequence and some \(\eps>0\) such that
\(\log(1/\theta_j)\ge\eps\gamma_j^2\). Then \(\theta_j\to0\), and \eqref{eq:energyidentity} and \eqref{eq:massbounds} give
\[
 E_j\le\gamma_jM_j\le
 \frac{C_\Om\gamma_j}{\sqrt{\log(1/\theta_j)}}\le C_\Om/\sqrt\eps.
\]
By Theorem~\ref{tool:DTglobal}, a further subsequence satisfies \(\sqrt{\theta_j}\gamma_j\to c>0\). On the other hand, the assumed parameter bound gives
\(\sqrt{\theta_j}\gamma_j\le\gamma_je^{-\eps\gamma_j^2/2}\to0\), a contradiction.
\end{proof}

\begin{lemma}\label{lem:maxbubble}
Suppose that \(u_j(x_j)=\gamma_j\to\infty\), and define \(\mu_j=(\theta_j\gamma_j^2e^{\gamma_j^2})^{-1/2}\). Then
\begin{equation}\label{eq:maxbubble}
 \gamma_j(u_j(x_j+\mu_jz)-\gamma_j)\to U(z)
 \quad\hbox{in }C^2_{\mathrm{loc}}(\R^2).
\end{equation}
For every fixed \(\sigma\in(0,1/2)\),
\begin{equation}\label{eq:maxbudget}
 \gamma_j\int_{B_{\exp(-\sigma\gamma_j^2)}(x_j)}\rho_j\dd x\to4\pi,
 \qquad
 \int_{B_{\exp(-\sigma\gamma_j^2)}(x_j)}u_j\rho_j\dd x\to4\pi.
\end{equation}
\end{lemma}
\begin{proof}
By Theorem~\ref{tool:boundary}, the points \(x_j\) lie in a fixed compact subset of \(\Om\). Lemma~\ref{lem:logtheta} gives \(\log(1/\mu_j)/\gamma_j^2\to1/2\), so the rescaled domains exhaust \(\R^2\).
The functions \(\eta_j=\gamma_j(u_j(x_j+\mu_jz)-\gamma_j)\) satisfy
\[
 \eta_j\le0=\eta_j(0),\qquad
 0\le-\Delta\eta_j=
 \left(1+\frac{\eta_j}{\gamma_j^2}\right)
 e^{2\eta_j+\eta_j^2/\gamma_j^2}\le1.
\]
The local \(W^{2,q}\) and Schauder estimates at the beginning of the section give a subsequential limit \(\eta\) satisfying \(-\Delta\eta=e^{2\eta}\) and \(\eta\le0=\eta(0)\).
In the Green representation formula at \(x_j\), retain only the integral over a fixed rescaled ball to obtain
\[
 1\ge\int_{B_T}\frac{G(x_j,x_j+\mu_jz)}{\gamma_j^2}
 \frac{u_j(x_j+\mu_jz)}{\gamma_j}
 e^{u_j(x_j+\mu_jz)^2-\gamma_j^2}\dd z.
\]
where
\[
 \frac{G(x_j,x_j+\mu_jz)}{\gamma_j^2}
 =\frac{\log(1/\mu_j)}{2\pi\gamma_j^2}
  -\frac{\log|z|}{2\pi\gamma_j^2}
  -\frac{H(x_j,x_j+\mu_jz)}{\gamma_j^2},
\]
The first term tends to \(1/(4\pi)\); the other two tend to zero uniformly on each fixed annulus \(\{\eps\le|z|\le T\}\), \(0<\eps<T\).
First exclude a small ball about the origin, use positivity to take the lower limit, and then shrink the excluded ball. This gives
\(1\ge(4\pi)^{-1}\int_{B_T}e^{2\eta}\). Let \(T\to\infty\). Theorem~\ref{tool:chenli}, together with \(\eta\le0=\eta(0)\), yields \(\eta=U\). Every subsequence has a further convergent subsequence with this limit, so \eqref{eq:maxbubble} holds for the full sequence.

Set \(m_T=\int_{B_T}e^{2U}\uparrow4\pi\) and \(r_j=e^{-\sigma\gamma_j^2}\). Under the change of variables \(y=x_j+\mu_jz\), we have
\[
 \begin{aligned}
 \gamma_j\mu_j^2\rho_j(x_j+\mu_jz)
 &=\frac{u_j(x_j+\mu_jz)}{\gamma_j}
   e^{u_j(x_j+\mu_jz)^2-\gamma_j^2}\to e^{2U(z)},\\
 \mu_j^2u_j\rho_j(x_j+\mu_jz)
 &=\left(\frac{u_j(x_j+\mu_jz)}{\gamma_j}\right)^2
   e^{u_j(x_j+\mu_jz)^2-\gamma_j^2}\to e^{2U(z)}.
 \end{aligned}
\]
In the second line, \(u_j\) is also evaluated at \(x_j+\mu_jz\). These densities are uniformly bounded on fixed balls.
Since \(\int_{B_T}|\log|z||\dd z<\infty\), we may also pass to the limit in the integral with the Green kernel. Thus
\[
 \gamma_j\int_{B_{T\mu_j}(x_j)}\rho_j\to m_T,
 \quad \int_{B_{T\mu_j}(x_j)}u_j\rho_j\to m_T,
 \quad \frac1{\gamma_j}\int_{B_{T\mu_j}(x_j)}G(x_j,y)\rho_j(y)\dd y\to\frac{m_T}{4\pi}.
\]
On \(B_{r_j}(x_j)\setminus B_{T\mu_j}(x_j)\), the Green function is at least \(\sigma\gamma_j^2/(2\pi)-C\). Subtracting the inner contribution from the Green representation formula and discarding the nonnegative outer contribution gives
\[
 \limsup_j\gamma_j\int_{B_{r_j}(x_j)\setminus B_{T\mu_j}(x_j)}\rho_j
 \le\frac{2\pi}{\sigma}\left(1-\frac{m_T}{4\pi}\right).
\]
Letting \(T\to\infty\) proves the first assertion of \eqref{eq:maxbudget}. Since \(u_j\le\gamma_j\), the upper limit of the weighted integral is at most \(4\pi\), while convergence of the integral on small balls gives a lower limit of at least \(4\pi\). This proves the second assertion.
\end{proof}

\begin{proposition}\label{prop:maxbound}
Every solution satisfies
\begin{equation}\label{eq:maxbound}
 \theta\gamma^2\le C_\Om,
\end{equation}
and
\begin{equation}\label{eq:logenergy}
 E\le\gamma M\le C'_\Om\bigl[1+\log^+(1/\theta)\bigr].
\end{equation}
In particular, \(E_j\to\infty\) implies \(\theta_j\to0\).
\end{proposition}
\begin{proof}
If \(\gamma_j\to\infty\), Lemmas~\ref{lem:logtheta} and \ref{lem:maxbubble} verify all hypotheses of Lemma~\ref{lem:budget}, including \eqref{eq:budget} with \(\sigma=1/32\). Taking \(\tau=1/4\) in Lemma~\ref{lem:matching}, we obtain
\begin{equation}\label{eq:maxmatching}
 \log(\theta_j\gamma_j^2/4)
 =4\pi R(x_j)-1-\gamma_j I_{\mathrm{out},j}(e^{-\gamma_j^2/4})+o(1).
\end{equation}
Since \(G\ge0\) and \(\rho_j\ge0\), we have \(I_{\mathrm{out},j}(e^{-\gamma_j^2/4})\ge0\). The function \(R\) is bounded on the fixed compact set containing all maximum points, so \(\theta_j\gamma_j^2\) is bounded. {If \eqref{eq:maxbound} failed for the full family of solutions, we could choose a sequence with \(\theta_j\gamma_j^2\to\infty\). Since \(\theta_j<\lambda_1(\Om)\), this would force \(\gamma_j\to\infty\), contradicting the preceding bound.}

By \eqref{eq:strip}, there is a compact set \(K_1\Subset\Om\), independent of the solution, such that \(\int_{K_1}\rho\dd x\ge M/2\). Recall that \(K_{\mathrm{peak}}\Subset\Om\) is the fixed compact set from Theorem~\ref{tool:boundary}, depending only on \(\Om\) and containing all critical points of every positive solution, in particular each maximum point \(x_j\). The Green function has a positive lower bound \(g_0\) on \(K_{\mathrm{peak}}\times K_1\): away from the diagonal this follows from compactness and positivity, while near the diagonal \(G\) tends to \(+\infty\). Abbreviate \(Q_j=Q_j(e^{-\gamma_j^2/4})\) and \(I_{\mathrm{out},j}=I_{\mathrm{out},j}(e^{-\gamma_j^2/4})\). Then
\[
 I_{\mathrm{out},j}\ge g_0(M_j/2-Q_j),\qquad
 \gamma_jQ_j=4\pi+o(1).
\]
Substituting this lower bound into \eqref{eq:maxmatching} and rearranging gives, for all sufficiently large \(j\),
\[
 \begin{aligned}
 \frac{g_0}2\gamma_jM_j
 &\le\gamma_j I_{\mathrm{out},j}+g_0\gamma_jQ_j\\
 &=4\pi R(x_j)-1-\log(\theta_j\gamma_j^2/4)
   +g_0\gamma_jQ_j+o(1)\\
 &\le C+\log^+(1/(\theta_j\gamma_j^2)).
 \end{aligned}
\]
Thus
    \[\gamma_jM_j\le C_\Om[1+\log^+(1/(\theta_j\gamma_j^2))].\]
{This estimate holds for all solutions with a constant depending only on \(\Om\). Otherwise, choose solutions for which
\[
 \frac{\gamma_jM_j}{1+\log^+(1/(\theta_j\gamma_j^2))}\longrightarrow\infty.
\]
Since \(M_j\le C_\Om\), necessarily \(\gamma_j\to\infty\), contradicting the estimate just proved.}

If \(\gamma\ge1\), then \(\log^+(1/(\theta\gamma^2))\le\log^+(1/\theta)\). If \(0<\gamma<1\), the mass estimate gives \(\gamma M\le M\le C_\Om\). Together with \(E\le\gamma M\), these bounds yield \eqref{eq:logenergy}.

 Finally, \(E_j\to\infty\) implies \(\gamma_j\to\infty\), and \eqref{eq:maxbound} then gives \(\theta_j\to0\).
\end{proof}

\section{Atomic concentration and decay of the total mass}\label{sec:pressure}
Throughout this section, assume that \(\theta_j\to0\). By Proposition~\ref{prop:maxbound}, this holds for the sequences of unbounded energy in the proof of the main theorem. Having bounded the first factor in the energy estimate, we turn to the weak limits of the normalized measures to study \(M_j/\sqrt{\theta_j}\). We prove the existence of an atom of positive mass, determine the relation between the two limiting atomic masses, and describe the distribution of function values near an atom. We then use this distribution estimate to improve the decay of the total mass and obtain the parameter and energy estimates needed for rescaling.

\subsection{Weak compactness and atomic decomposition}
Taking \(D=\Om\) and \(\psi=1\) in \eqref{eq:pressure-split}, we obtain
\(\Pi_j(\Om)\le C_{L,\Om}\theta_j+L^{-1}\). Letting first \(j\to\infty\) and then \(L\uparrow4\pi\) gives
\begin{equation}\label{eq:pressuretotal}
 \sup_j\Pi_j(\Om)<\infty,\qquad
 \limsup_j\Pi_j(\Om)\le\frac1{4\pi}.
\end{equation}
\begin{lemma}\label{lem:boundary-zero}
Regard \(\nu_j,\Pi_j\) as measures on \(\cl\Om\), following the extension convention in Section~\ref{sec:notation}. There is a common subsequence such that
\[
 \nu_j\weak\nu,\qquad \Pi_j\weak\Pi
 \quad\hbox{on }\cl\Om.
\]
Every such pair of weak limits satisfies
\[
 \nu(\partial\Om)=\Pi(\partial\Om)=0,
 \qquad \nu(\Om)=1.
\]
In particular, \(\nu\) is a probability measure on \(\Om\).
\end{lemma}
\begin{proof}
Set \(K=\sup_j\Pi_j(\Om)<\infty\). Since \(\cl\Om\) is compact and
\(\nu_j(\cl\Om)=1\), \(\Pi_j(\cl\Om)\le K\), weak compactness of finite nonnegative measures gives a common subsequence such that
\(\nu_j\weak\nu\) and \(\Pi_j\weak\Pi\), where \(\nu,\Pi\) are finite nonnegative measures on \(\cl\Om\).

Let \(d\) be the distance function from Lemma~\ref{lem:strip}. Fix \(0<\eps<\delta_0\), and define
\[
 \chi_\eps(x)=\max\left\{1-\frac{d(x)}{\eps},0\right\},
 \qquad x\in\cl\Om.
\]
Then \(\chi_\eps\in C(\cl\Om)\), \(0\le\chi_\eps\le1\), and \(\chi_\eps\) equals \(1\) on \(\partial\Om\) and \(0\) on
\(\{d\ge\eps\}\). Testing weak convergence with \(\chi_\eps\) and using the boundary strip estimate \eqref{eq:strip}, we obtain
\[
 \begin{aligned}
  \nu(\partial\Om)
  &\le\int_{\cl\Om}\chi_\eps\dd\nu
  =\lim_{j\to\infty}\int_{\cl\Om}\chi_\eps\dd\nu_j
  \le C_\Om\eps,\\
  \Pi(\partial\Om)
  &\le\int_{\cl\Om}\chi_\eps\dd\Pi
  =\lim_{j\to\infty}\int_{\cl\Om}\chi_\eps\dd\Pi_j
  \le C_\Om K\eps.
 \end{aligned}
\]
Here \(\eps\) is fixed before \(j\to\infty\). Letting \(\eps\downarrow0\) gives
\(\nu(\partial\Om)=\Pi(\partial\Om)=0\). Finally, testing weak convergence with the constant function \(1\) yields
\[
 \nu(\cl\Om)=\lim_{j\to\infty}\nu_j(\cl\Om)=1.
\]
Thus \(\nu(\Om)=\nu(\cl\Om)-\nu(\partial\Om)=1\).
\end{proof}

From now on, fix the common subsequence and its weak limits \(\nu,\Pi\) given by Lemma~\ref{lem:boundary-zero}.

\begin{lemma}\label{lem:pressureupper}
Let \(\{a_i\}\) be the set of all atoms of \(\nu\), namely the points for which \(p_i:=\nu(\{a_i\})>0\), and set \(\tau_i:=\Pi(\{a_i\})\). Then
\begin{equation}\label{eq:pressureupper}
 \Pi=\sum_i\tau_i\delta_{a_i},\qquad 0\le\tau_i\le\frac{p_i^2}{4\pi}.
\end{equation}
The set of atoms is at most countable.
\end{lemma}
\begin{proof}
Fix \(B_{2r}(x_0)\Subset\Om\), \(\bar p>\nu(\cl B_{2r}(x_0))\), and
\(0<L<4\pi/\bar p\). For a nonnegative \(\psi\in C_c(B_r(x_0))\), letting \(j\to\infty\) in \eqref{eq:pressure-split} gives
\[
 \int\psi\dd\Pi\le L^{-1}\int\psi\dd\nu.
\]
Letting first \(L\uparrow4\pi/\bar p\) and then \(\bar p\downarrow\nu(\cl B_{2r}(x_0))\), we obtain the local measure inequality
\begin{equation}\label{eq:local-pressure-domination}
 \Pi|_{B_r(x_0)}
 \le\frac{\nu(\cl B_{2r}(x_0))}{4\pi}\,\nu|_{B_r(x_0)}.
\end{equation}
This argument also covers the case in which the coefficient on the right is zero. Similarly, in the global version of \eqref{eq:pressure-split}, letting first \(j\to\infty\) and then \(L\uparrow4\pi\) yields \(\Pi\le\nu/(4\pi)\). Thus \(\Pi\ll\nu\). By the Radon--Nikodym theorem, we may write
\(\Pi=f\nu\), where \(f\) is a nonnegative Borel function and \(0\le f\le1/(4\pi)\) holds \(\nu\)-almost everywhere.

To make all local upper bounds hold outside a single null set, consider the countable family of balls
\[
 \mathcal B=\{B_r(x_0):x_0\in\mathbb Q^2,\ r\in\mathbb Q,
 \ r>0,\ \cl B_{2r}(x_0)\subset\Om\}.
\]
For each \(B=B_r(x_0)\in\mathcal B\), \eqref{eq:local-pressure-domination} implies
\(f\le c_B:=\nu(\cl B_{2r}(x_0))/(4\pi)\) \(\nu\)-almost everywhere on \(B\).
Indeed, applying the measure inequality to \(F_m=\{x\in B:f(x)>c_B+1/m\}\) gives \(\nu(F_m)=0\).
Taking the union over all \(B\in\mathcal B\) and \(m\ge1\), we obtain a \(\nu\)-null set \(N\) outside which all these bounds hold simultaneously.

Fix \(x\notin N\). Choose rational numbers \(r_n\downarrow0\) and points \(x_n\in\mathbb Q^2\) such that
\(3r_n<\dist(x,\partial\Om)\) and \(|x_n-x|<r_n/2\). Then
\(x\in B_{r_n}(x_n)\in\mathcal B\), and
\[
 0\le f(x)\le\frac{\nu(\cl B_{2r_n}(x_n))}{4\pi}
 \le\frac{\nu(\cl B_{3r_n}(x))}{4\pi}
 \longrightarrow\frac{\nu(\{x\})}{4\pi}.
\]
The last step uses continuity from above for finite measures. Thus \(f\) vanishes almost everywhere away from the atoms.

The set \(\mathcal A=\{x:\nu(\{x\})>0\}\) is at most countable: for each \(m\ge1\), there are at most \(m\) points with mass at least \(1/m\). Hence
\[
 \Pi(\Om\setminus\mathcal A)=\int_{\Om\setminus\mathcal A}f\dd\nu=0,
 \qquad \Pi=\sum_i\tau_i\delta_{a_i}.
\]
Every atom lies outside \(N\), so
\(0\le\tau_i=f(a_i)p_i\le p_i^2/(4\pi)\), proving \eqref{eq:pressureupper}.
\end{proof}

\subsection{An integral identity from the Green representation formula}
Lemma~\ref{lem:pressureupper} bounds the mass of \(\Pi\) at each atom of \(\nu\). To determine the exact value, we derive an integral identity from the Green representation formula and apply it near an atom.
For \(X\in C_c^\infty(\Om;\R^2)\) and a finite nonnegative Borel measure \(\mu\) on \(\Om\), define the following quantities, with \(K_X\) defined only for \(x\ne y\):
\begin{equation}\label{eq:stresskernels}
 K_X(x,y):=\frac{(X(x)-X(y))\cdot(x-y)}{|x-y|^2},
\end{equation}
\begin{equation}
 B_X(\mu):=\iint_{\Omega\times\Omega} X(x)\cdot\nabla_xH(x,y)\dd\mu(x)\dd\mu(y).
\end{equation}
Here \(\Lip X=\sup_{x\ne y}|X(x)-X(y)|/|x-y|\), so \(|K_X|\le\Lip X\). When \(X\) is compactly supported, the integrand in \(B_X\) is bounded and continuous. For fixed \(j\), the measure \(\nu_j\) is absolutely continuous, and the diagonal has zero \(\nu_j\otimes\nu_j\) measure. The value assigned to \(K_X\) on the diagonal therefore does not affect the integral. The limiting measure may have atoms, whose contribution on the diagonal must be treated separately.

\begin{lemma}\label{lem:stress}
For every \(j\),
\begin{equation}\label{eq:stress}
 \frac1{4\pi}\iint_{\Omega\times\Omega} K_X\dd\nu_j\dd\nu_j+B_X(\nu_j)
 =\frac12\int_{\Omega}\diver X\dd\Pi_j.
\end{equation}
The identity also holds for compactly supported Lipschitz vector fields, with divergence understood almost everywhere.
\end{lemma}
\begin{proof}
First suppose that \(X\in C_c^\infty(\Om;\R^2)\). The chain rule gives
\(\nabla P_j=2\theta_ju_je^{u_j^2}\nabla u_j=2\rho_j\nabla u_j\). Hence
    \begin{equation}\label{eq:intpho_jXcdotnablau_j}
    	\int_\Om\rho_jX\cdot\nabla u_j
    	=\frac12\int_\Om X\cdot\nabla P_j
    	=-\frac12\int_\Om P_j\diver X.
    \end{equation}
Differentiating the Green representation formula and using \eqref{eq:green}, we obtain
\[
 \nabla u_j(x)=-\frac1{2\pi}\int_\Om\frac{x-y}{|x-y|^2}\rho_j(y)\dd y
 -\int_\Om\nabla_xH(x,y)\rho_j(y)\dd y.
\]
For fixed \(j\), \(\rho_j\) is bounded, the kernel \(|x-y|^{-1}\) is locally integrable in dimension two, and the regular term is uniformly bounded for
\(x\in\supp X\). Thus the double integrals below converge absolutely, and Fubini's theorem applies.
Interchanging \(x,y\) in the singular term gives
\[
 \begin{aligned}
 &\iint_{\Omega\times\Omega}\frac{X(x)\cdot(x-y)}{|x-y|^2}\rho_j(x)\rho_j(y)\dd x\dd y\\
 &\qquad=-\iint_{\Omega\times\Omega}\frac{X(y)\cdot(x-y)}{|x-y|^2}\rho_j(x)\rho_j(y)\dd x\dd y\\
 &\qquad=\frac12\iint_{\Omega\times\Omega} K_X(x,y)\rho_j(x)\rho_j(y)\dd x\dd y.
 \end{aligned}
\]
The last equality is obtained by averaging the first two expressions. Therefore
    \begin{align}
    	\int_\Om\rho_jX\cdot\nabla u_j
    	&=-\frac1{4\pi}\iint_{\Omega\times\Omega} K_X\rho_j(x)\rho_j(y)\dd x\dd y\\
    	&\qquad-\iint_{\Omega\times\Omega} X(x)\cdot\nabla_xH(x,y)\rho_j(x)\rho_j(y)\dd x\dd y.
    \end{align}
Comparing this with the integration-by-parts identity \eqref{eq:intpho_jXcdotnablau_j}, cancelling the minus signs, and dividing by \(M_j^2\) proves \eqref{eq:stress}.

For a compactly supported Lipschitz vector field \(X\), extend it by zero and mollify to obtain \(X_\ell\). The identity \eqref{eq:stress} holds for \(X_\ell\). Their supports can be kept in a fixed compact subset of \(\Om\), with \(X_\ell\to X\) uniformly, \(\Lip X_\ell\le\Lip X\), and derivatives converging almost everywhere. For \(x\ne y\), we have \(K_{X_\ell}(x,y)\to K_X(x,y)\) and \(|K_{X_\ell}|\le\Lip X\). Dominated convergence therefore extends \eqref{eq:stress} to Lipschitz vector fields.
\end{proof}

\subsection{Atomic mass relations and existence of atoms}
\begin{proposition}\label{prop:square}
The weak limit \(\Pi\) is given by
\begin{equation}\label{eq:square}
 \Pi=\frac1{4\pi}\sum_i p_i^2\delta_{a_i}.
\end{equation}
\end{proposition}
\begin{proof}
Fix an atom \(a\) of \(\nu\), and write \(p=\nu(\{a\})>0\), \(\tau=\Pi(\{a\})\).
Choose \(\chi\in C_c^\infty([0,\infty))\) equal to \(1\) on \([0,1]\) and to \(0\) on \([2,\infty)\), and set
\[
 X_\eps(x)=(x-a)\chi(|x-a|/\eps).
\]
For sufficiently small \(\eps\), this vector field is compactly supported in \(\Om\). Direct differentiation gives
\(\|X_\eps\|_\infty\le C\eps\) and \(\Lip X_\eps\le C\), with \(C\) independent of \(\eps\).
If \(x,y\in B_\eps(a)\) and \(x\ne y\), then \(X_\eps(x)-X_\eps(y)=x-y\), so \(K_{X_\eps}(x,y)=1\).

To estimate the singular term, take \(0<\eta<\eps\) with \(\nu(\partial B_\eta(a))=\nu(\partial B_{2\eps}(a))=0\). This excludes at most countably many radii. Write
\(I_{j,\eps}=\iint_{\Omega\times\Omega} K_{X_\eps}\dd\nu_j\dd\nu_j\).
On \(B_\eta(a)^2\), the kernel \(K_{X_{\varepsilon}}\) equals \(1\), and the integral is exactly \(\nu_j(B_\eta(a))^2\).
When at least one variable belongs to \(B_{2\eps}(a)\setminus B_\eta(a)\),
the bounds \(|K_{X_\eps}|\le C\) and \(\nu_j(\Om)=1\) show that the absolute value of the contribution is at most
\(2C\nu_j(B_{2\eps}(a)\setminus B_\eta(a))\).
If \(x\in B_\eta(a)\) and \(y\notin B_{2\eps}(a)\), then
\[
 |K_{X_\eps}(x,y)|
 =\left|\frac{(x-a)\cdot(x-y)}{|x-y|^2}\right|
 \le\frac{\eta}{2\eps-\eta}.
\]
The same bound holds after interchanging the variables. When both variables lie outside \(B_{2\eps}(a)\), the kernel vanishes. Thus
\[
 \left|I_{j,\eps}-\nu_j(B_\eta(a))^2\right|
 \le2C\nu_j(B_{2\eps}(a)\setminus B_\eta(a))
   +\frac{2\eta}{2\eps-\eta}.
\]
Letting first \(j\to\infty\) and then \(\eta\downarrow0\), we obtain
\[
 \limsup_j|I_{j,\eps}-p^2|
 \le2C\bigl(\nu(B_{2\eps}(a))-p\bigr).
\]
Finally, let \(\eps\downarrow0\) through admissible radii; the right-hand side tends to zero.

For the regular term in \eqref{eq:stress}, \eqref{eq:greenbounds} gives
\[
 |B_{X_\eps}(\nu_j)|\le C\|X_\eps\|_\infty\nu_j(\Om)^2\le C\eps.
\]
For fixed \(\eps\), the function \(\diver X_\eps\) is continuous and compactly supported, so weak convergence of \(\Pi_j\) yields
\(\int\diver X_\eps\dd\Pi_j\to\int\diver X_\eps\dd\Pi\). Moreover,
\[
 \diver X_\eps(x)
 =2\chi(|x-a|/\eps)+\frac{|x-a|}{\eps}\chi'(|x-a|/\eps),
\]
This expression is uniformly bounded in absolute value, equals \(2\) at \(a\), and eventually vanishes at each \(x\ne a\).
Dominated convergence therefore gives
\(\frac12\int\diver X_\eps\dd\Pi\to\Pi(\{a\})=\tau\).
Substituting these limits into \eqref{eq:stress} yields \(p^2/(4\pi)=\tau\).
Applying this equality to every atom and using the atomic decomposition from Lemma~\ref{lem:pressureupper} proves \eqref{eq:square}.
\end{proof}

\begin{proposition}\label{prop:atom}
At least one atom satisfies \(p_i\ge G_*^{-1}>0\), where \(G_*\) is defined in \eqref{eq:v-a-h-g-G}.
\end{proposition}
\begin{proof}
Pull the normalized measures back to \(\D\), writing \(\nu_j^\D=(\phi^{-1})_\#\nu_j\) and \(\Pi_j^\D=(\phi^{-1})_\#\Pi_j\). For example, \(\nu_j^\D(B)=\nu_j(\phi(B))\) for each Borel set \(B\subset\D\). Define \(\nu^\D,\Pi^\D\) similarly. Set \(q_{n,j}=\int_\D z^n\dd\nu_j^\D\). We first estimate their uniform decay as \(n\) increases.

The map \(\phi\) is Lipschitz on the closed disk. For \(z\ne0\), comparison with the boundary point \(\phi(z/|z|)\) gives
\[
 \dist(\phi(z),\partial\Om)\le C\lvert\phi(z)-\phi(z/\lvert z\rvert)\rvert\leq C(1-|z|).
\]
Hence Lemma~\ref{lem:strip} implies \(\nu_j^\D\{|z|>t\}\le C(1-t)\) for \(t\) close to \(1\).
For \(t\) bounded away from \(1\), the same estimate follows from unit total mass after enlarging the constant.
Since \(|z|^n=\int_0^1nt^{n-1}\ind_{\{|z|>t\}}\dd t\), Tonelli's theorem gives
\[
 |q_{n,j}|\le\int_\D |z|^n\dd\nu_j^\D
 =\int_0^1 nt^{n-1}\nu_j^\D\{|z|>t\}\dd t
 \le\frac C{n+1}.
\]
Thus the tails of \(\sum_{n\ge1}|q_{n,j}|^2\) tend to zero uniformly in \(j\). For fixed \(n\), \(q_{n,j}\) converges as \(j\to\infty\). Dividing \eqref{eq:moment} by \(M_j^2\), using \eqref{eq:square}, and passing to the limit, we obtain
\begin{equation}\label{eq:momentlimit}
 1+2\sum_{n\ge1}\left|\int_\D z^n\dd\nu^\D\right|^2
 =\sum_i g(\phi^{-1}(a_i))p_i^2
 \le G_*\max_i p_i.
\end{equation}
The equality, whose left-hand side is at least \(1\), first shows that the set of atoms is nonempty. Since \(\sum_i p_i\le1\), only finitely many masses exceed \((\sup_i p_i)/2\), so the maximum is attained. The final inequality uses \(\sum_i p_i^2\le(\max_i p_i)\sum_i p_i\) and gives the stated lower bound.
\end{proof}

\subsection{Distribution of values near an atom}
We determine the distribution of \(s_j\) near an atom by first letting \(j\to\infty\) and then shrinking its neighborhood.
\begin{proposition}\label{prop:height}
Fix an atom \(a\) of mass \(p>0\), and set \(A=4\pi/p\). For every \(\eps>0\),
\begin{equation}\label{eq:height}
 \lim_{r\downarrow0}\limsup_{j\to\infty}
 \nu_j\bigl(B_r(a)\cap\{|s_j-A|>\eps\}\bigr)=0.
\end{equation}
In this limit, \(r\) may be restricted to a sequence tending to zero with \(\nu(\partial B_r(a))=\Pi(\partial B_r(a))=0\). In addition,
\begin{equation}\label{eq:TVheight}
 \lim_{r\downarrow0}\limsup_j
 \left\|\Pi_j-\frac1A\nu_j\right\|_{\TV(B_r(a))}=0.
\end{equation}
The total variation norm is defined in Section~\ref{sec:notation}.
\end{proposition}
\begin{proof}
\emph{Values below \(A\).}
Fix \(0<b<A\). Since \(\nu(\cl B_{2r}(a))\downarrow p\) and \(bp<4\pi\),
we can first choose \(r\) sufficiently small and then \(\bar p>\nu(\cl B_{2r}(a))\) such that \(b\bar p<4\pi\).
Choose \(\alpha\in(b\bar p,4\pi)\). Lemma~\ref{lem:jensen} gives
\[
 \nu_j(B_r(a)\cap\{s_j\le b\})+
 \Pi_j(B_r(a)\cap\{s_j\le b\})\le C_{r,b,\bar p,\alpha,\Omega}\theta_j\to0.
\]
For \(0<\eps<A\), take \(b=A-\eps\). The \(\nu_j\) measure of \(B_r(a)\cap\{s_j<A-\eps\}\) then tends to zero in the stated order of limits.
If \(\eps\ge A\), the set \(\{s_j<A-\eps\}\) is empty because \(s_j>0\).

\emph{Values above \(A\).}
Fix \(\eps>0\), set \(h=A+\eps\), and take any \(0<b<A\). Write
\(t_{j,r}=\nu_j(B_r(a)\cap\{s_j>h\})\). By \eqref{eq:pressure-comparison}, \(\dd\Pi_j/\dd\nu_j\le1/b\) on \(\{b<s_j\le h\}\), whereas this ratio is at most \(1/h\) on \(\{s_j>h\}\).
Combining this with the estimate above on \(\{s_j\le b\}\), we obtain, for fixed sufficiently small \(r\),
\[
 \begin{aligned}
 \Pi_j(B_r(a))
 &\le o_j(1)+\frac1b\nu_j(B_r(a)\cap\{b<s_j\le h\})+\frac{t_{j,r}}h\\
 &\le o_j(1)+\frac{\nu_j(B_r(a))}b
       -\left(\frac1b-\frac1h\right)t_{j,r}.
 \end{aligned}
\]
The coefficient \(1/b-1/h\) is positive. Move the term containing \(t_{j,r}\) to the left and let first \(j\to\infty\),
then \(r\downarrow0\) through radii satisfying \(\nu(\partial B_r(a))=\Pi(\partial B_r(a))=0\). By \eqref{eq:square},
\(\Pi(\{a\})=p^2/(4\pi)=p/A\), so
\[
 \lim_{r\downarrow0}\limsup_jt_{j,r}
 \le\frac{p/b-p/A}{1/b-1/h}.
\]
Finally, let \(b\uparrow A\). The numerator tends to zero, whereas the denominator tends to \(1/A-1/h>0\); hence the right-hand side tends to zero.
Together with the estimate for \(\{s_j<A-\eps\}\), this proves \eqref{eq:height}. As \(r\) decreases, the \(\nu_j\) measure of \(B_r(a)\cap\{|s_j-A|>\eps\}\) is nonincreasing.
Consequently, the limit as \(r\downarrow0\) is zero without restricting the choice of radii.

\emph{Convergence in total variation.}
Set \(h_j=(1-e^{-u_j^2})/s_j\). By \eqref{eq:ratios},
\[
 \left\|\Pi_j-\frac1A\nu_j\right\|_{\TV(B_r(a))}
 =\int_{B_r(a)}|h_j-A^{-1}|\dd\nu_j.
\]
Fix \(b\in(0,A)\), and choose \(0<\eps<A-b\). On \(\{s_j\le b\}\), the integral is at most
\[
 \Pi_j(B_r(a)\cap\{s_j\le b\})
 +A^{-1}\nu_j(B_r(a)\cap\{s_j\le b\}),
\]
which tends to zero by the estimate above on \(\{s_j\le b\}\). On \(\{s_j>b\}\cap\{|s_j-A|>\eps\}\),
we have \(|h_j-A^{-1}|\le b^{-1}+A^{-1}\). By \eqref{eq:height}, the \(\nu_j\) mass of this set in \(B_r(a)\) tends to zero as first \(j\to\infty\) and then \(r\downarrow0\); so does the corresponding integral.
On the remaining set \(\{|s_j-A|\le\eps\}\), use \(u_j=s_j/M_j\) and \(M_j\to0\), which follows from \eqref{eq:massbounds}, to obtain
\[
 |h_j-A^{-1}|
 \le\left|\frac1{s_j}-\frac1A\right|+\frac{e^{-s_j^2/M_j^2}}{s_j}
 \le\frac{\eps}{A(A-\eps)}+\frac1b e^{-b^2/M_j^2}.
\]
Since \(\nu_j(\Om)=1\), letting first \(j\to\infty\) and then \(r\downarrow0\) bounds the upper limit of the total variation by
\(\eps/[A(A-\eps)]\). Letting \(\eps\downarrow0\) proves \eqref{eq:TVheight}.
\end{proof}

\subsection{An improved decay estimate for the total mass}\label{sec:strict}

When \(\theta_j\to0\), \eqref{eq:massbounds} shows that \(M_j^2\log(1/\theta_j)\) is bounded. We use Proposition~\ref{prop:height} to prove that this quantity tends to zero. Together with the logarithmic energy bound in Section~\ref{sec:local}, this gives \(M_j^2E_j\to0\). Both estimates will be used in Section~\ref{sec:center} to compare the rescaling parameters and control local oscillation. We first estimate the Green potential of a bounded density.

\begin{lemma}\label{lem:densitypotential}
Let \(f\in L^\infty(\Om)\), \(0\le f\le D\), and set \(\dd\eta=f\dd x\), \(m=\eta(\Om)>0\). Then its Green potential satisfies
\begin{equation}\label{eq:densitypotential}
 \sup_{y\in\Om}G\eta(y)\le\frac m{4\pi}\left[\log\frac{C_\Om D}{m}+1\right].
\end{equation}
\end{lemma}
\begin{proof}
Extend \(f\) by zero outside \(\Om\), fix \(y\in\Om\), and write
\[
 r=\sqrt{\frac{m}{\pi D}},\qquad
 f_*(x)=D\ind_{B_r(y)}(x),\qquad
 k_y(x)=\log\frac{L_0}{|x-y|}.
\]
Both densities have total mass \(m\). On \(B_r(y)\), we have \(f-f_*\le0\) and
\(k_y-\log(L_0/r)\ge0\); outside the ball, the signs are reversed. Hence
\[
 \begin{aligned}
 \int_{\R^2}k_y(f-f_*)\dd x
 &=\int_{\R^2}\left(k_y-\log\frac{L_0}{r}\right)(f-f_*)\dd x\le0.
 \end{aligned}
\]
The equality uses \(\int(f-f_*)=0\). The logarithmic singularity is integrable for bounded densities, so these integrals are well defined.
By \eqref{eq:greenbounds}, the Green potential is bounded above by \((2\pi)^{-1}\int k_yf\),
which is no larger than the value for the {density} \(f_*\). Integration in polar coordinates gives
\[
 \begin{aligned}
 \frac D{2\pi}\int_{B_r(0)}\log\frac{L_0}{|z|}\dd z
 &=D\int_0^r t\log\frac{L_0}{t}\dd t\\
 &=D\left(\frac{r^2}2\log\frac{L_0}{r}+\frac{r^2}4\right)\\
 &=\frac m{4\pi}\left[\log\frac{\pi L_0^2D}{m}+1\right].
 \end{aligned}
\]
Since \(m/D\le|\Om|\), we may choose \(L_0\) large enough such that \(r<L_0\). Taking \(C_\Om=\pi L_0^2\) proves \eqref{eq:densitypotential}.
\end{proof}

\begin{proposition}\label{prop:strict}
Every sequence with \(\theta_j\to0\) satisfies
\begin{equation}\label{eq:strict}
 M_j^2\log(1/\theta_j)\to0,\qquad M_j^2E_j\to0,
 \qquad s_j\to0\quad\hbox{strongly in }H^1_0(\Om).
\end{equation}
\end{proposition}
\begin{proof}
By \eqref{eq:massbounds}, \(b_j=M_j^2\log(1/\theta_j)\) is bounded and eventually nonnegative. From any subsequence, extract a further one such that \(b_j\to b\ge0\) and \(\nu_j,\Pi_j\) converge weakly to \(\nu,\Pi\), respectively. Proposition~\ref{prop:atom} gives an atom \(a\) of mass \(p>0\); set \(A=4\pi/p\).
Fix \(0<\eps<A\), choose \(r\) sufficiently small with both \(\nu\) and \(\Pi\) vanishing on \(\partial B_r(a)\) and \(\partial B_{2r}(a)\), and define
\[
 \eta_j=\nu_j|_{B_r(a)\cap\{|s_j-A|\le\eps\}},\qquad m_j=\eta_j(\Om).
\]
By Proposition~\ref{prop:height}, \(\liminf_jm_j\ge p/2\) for sufficiently small \(r\). {Set \(F_j=\Om\cap B_r(a)\cap\{|s_j-A|\le\eps\}\). Using \(\rho_j=\theta_ju_je^{u_j^2}\) and \(u_j=s_j/M_j\), the density of the restricted measure is
\[
 \frac{\dd\eta_j}{\dd x}
 =\ind_{F_j}\frac{\rho_j}{M_j}
 =\ind_{F_j}\frac{\theta_js_j}{M_j^2}
  \exp\left(\frac{s_j^2}{M_j^2}\right).
\]
On \(F_j\), we have \(0<A-\eps\le s_j\le A+\eps\), while this density vanishes outside \(F_j\). Thus, almost everywhere in \(\Om\),
\[
 0\le\frac{\dd\eta_j}{\dd x}\le D_j:=
 \frac{\theta_j(A+\eps)}{M_j^2}
 \exp\left(\frac{(A+\eps)^2}{M_j^2}\right).
\]}
Symmetry of the Green function and \(s_j=M_j^2G\nu_j\) give
\[
 (A-\eps)m_j\le\int s_j\dd\eta_j
 =M_j^2\int G\eta_j\dd\nu_j.
\]
Split \(\int_\Om G\eta_j\dd\nu_j\) into the parts inside and outside \(B_{2r}(a)\). Inside the ball, apply Lemma~\ref{lem:densitypotential} to \(G\eta_j\). Outside, the distance to \(\supp\eta_j\) is at least \(r\), so \(G\eta_j\le C_rm_j\). Substitution into the preceding inequality and division by \(m_j>0\) yield
\begin{align*}
 A-\eps\le\frac{\nu_j(B_{2r}(a))}{4\pi}
 \biggl[(A+\eps)^2-b_j-2M_j^2\log M_j
 +M_j^2\log\frac{C_\Om(A+\eps)}{m_j}+M_j^2\biggr]+C_rM_j^2.
\end{align*}
For fixed \(r,\eps\), we eventually have \(p/3\le m_j\le1\), while \(M_j\to0\) and \(M_j^2\log M_j\to0\). Using \(\nu(\partial B_{2r}(a))=0\) and letting \(j\to\infty\), we obtain
\[
 A-\eps\le\frac{\nu(B_{2r}(a))}{4\pi}\bigl((A+\eps)^2-b\bigr).
\]
Letting next \(r\downarrow0\) and then \(\eps\downarrow0\) gives
\[
 A\le\frac p{4\pi}(A^2-b)=A-\frac{pb}{4\pi},
\]
Thus \(b=0\). Every convergent subsequence has this limit, so the first assertion holds for the full sequence. 

\begin{revision}
Combining \(M_j\to0\) with \eqref{eq:logenergy}, we obtain, for all sufficiently large \(j\),
\[
 0\le M_j^2E_j
 \le C'_\Om\bigl[M_j^2+M_j^2\log(1/\theta_j)\bigr]\to0.
\]
\end{revision}
 Finally, \(\|\nabla s_j\|_2^2=M_j^2E_j\), and Poincar\'e's inequality yields the stated strong convergence.
\end{proof}

\section{Minimal radii and limits of rescaled measures}\label{sec:halfmass}

Let \(a_0\) be an atom of the weak limit \(\nu\), with \(p=\nu(\{a_0\})>0\). To determine the concentration scale of \(\nu_j\) near \(a_0\), for each sufficiently large \(j\) we select a ball of minimal radius satisfying \(\nu_j(B_{r_j}(b_j))=p/2\), with its center restricted to a prescribed set, and rescale by \(r_j\). {The mass of the ball is computed with respect to} \(\nu_j\), which has a continuous density; the threshold \(p/2\) is determined by the limiting atom. We first prove a classification lemma for the rescaled limiting measure.

\subsection{Classification of measures satisfying a planar integral identity}
In this subsection, the kernel \(K_X\) from \eqref{eq:stresskernels} is defined by the same formula off the diagonal in \(\R^2\times\R^2\).
\begin{lemma}\label{lem:weakclassify}
Let \(A>0\), and let \(\mu\ne0\) be a finite nonnegative nonatomic measure on \(\R^2\). Suppose that, for every \(X\in C_c^\infty(\R^2;\R^2)\),
\begin{equation}\label{eq:planarstress}
 \frac1{4\pi}\iint_{\R^2\times\R^2} K_X(z,y)\dd\mu(z)\dd\mu(y)
 =\frac1{2A}\int_{\R^2}\diver X\dd\mu.
\end{equation}
Then
\begin{equation}\label{eq:classdensity}
 \dd\mu(z)=\frac4A\,\frac{\lambda^2}{(1+\lambda^2|z-z_0|^2)^2}\dd z
 \quad\hbox{for some }\lambda>0,\ z_0\in\R^2.
\end{equation}
\end{lemma}
\begin{proof}
Fix \(z_*\in\R^2\), set \(m(r)=\mu(B_r(z_*))\), and define
\[
 X_r(z)=\frac{z-z_*}{\max\{r^2,|z-z_*|^2\}},
 \qquad k_r(z)=\frac1{\max\{r,|z-z_*|\}}.
\]
The map \(X_r\) is Lipschitz and, away from the circle of radius \(r\), its divergence is
\(2r^{-2}\ind_{B_r(z_*)}\). It satisfies
\begin{equation}\label{eq:inversion}
 |X_r(z)-X_r(y)|\le |z-y|k_r(z)k_r(y).
\end{equation}
Translate \(z_*\) to the origin. When both points lie inside the ball, \(X_r(z)=z/r^2\), so \eqref{eq:inversion} holds. When both lie outside, a direct calculation gives
\[
 \left|\frac{z}{|z|^2}-\frac{y}{|y|^2}\right|
 =\frac{|z-y|}{|z||y|},
\]
which is the equality case of \eqref{eq:inversion}. If one point is inside and the other outside, assume \(|z|\le r\le|y|\). Squaring and cancelling the {cross terms}, the square of the left-hand side minus the square of the right-hand side is
\((|z|^2-r^2)(r^{-4}-r^{-2}|y|^{-2})\le0\), proving the inequality.

Choose \(r\) with \(\mu(\partial B_r(z_*))=0\). Multiply \(X_r\) by a smooth cutoff equal to one on \(B_L(z_*)\) and supported in \(B_{2L}(z_*)\), and smooth it near the circle of radius \(r\). The divergences converge off that circle and are uniformly bounded for fixed \(r\), so their integrals {with respect to} \(\mu\) converge. The corresponding kernels \(K_X\) are uniformly bounded, and \(\mu\) has no atoms. These approximations extend \eqref{eq:planarstress} to the chosen vector field. We may then remove the cutoff: the divergence error is \(O(L^{-2})\), and the kernels remain uniformly bounded for fixed \(r\). By \eqref{eq:inversion},
\begin{equation}\label{eq:massode}
 \frac{m(r)}{Ar^2}\le\frac1{4\pi}F(r)^2,\qquad
 F(r)=\int_{\R^2} k_r\dd\mu=\int_r^\infty\frac{m(t)}{t^2}\dd t.
\end{equation}
where the expression for \(F\) follows by integrating the identity
\[
 \frac1{\max\{r,|z-z_*|\}}
 =\int_r^\infty\frac{\ind_{\{|z-z_*|<t\}}}{t^2}\dd t.
\]
{with respect to} \(\mu\) and applying Tonelli's theorem. The function \(F\) is absolutely continuous on every compact subinterval of \((0,\infty)\), with
\(F'(r)=-m(r)/r^2\) almost everywhere.
Set \(f(r)=rF(r)\) and \(C=A/(4\pi)\). Almost everywhere,
\[
 m(r)\le Cf(r)^2,\qquad rf'(r)=f(r)-m(r)\ge f(r)-Cf(r)^2.
\]
Since \(\mu\) has no atoms, \(f(r)\to0\). Indeed, splitting the defining integral over a ball of fixed small radius \(\delta\) and its complement gives
\(f(r)\le\mu(B_\delta(z_*))+r\mu(\R^2)/\delta\). Let first \(r\downarrow0\), then \(\delta\downarrow0\), and use \(\mu(\{z_*\})=0\).
Choose \(0<\eps<1/2\), and then \(r_0\) sufficiently small that \(Cf(r)\le\eps\) for \(r\le r_0\). Since \(f(r)>0\), integrating \(f'/f\ge(1-\eps)/r\) from \(r\) to \(r_0\) yields
\[
 f(r)\le f(r_0)(r/r_0)^{1-\eps},\qquad m(r)\le C_1r^{2-2\eps}.
\]

The constants can be chosen independently of the center. Indeed, a finite nonatomic measure satisfies
\(\sup_{z_*}\mu(\cl B_\delta(z_*))\to0\) as \(\delta\downarrow0\). Otherwise, there would be balls whose radii tend to zero but whose masses have a positive lower bound. Since \(\mu(\R^2\setminus B_R(0))\to0\) as \(R\to\infty\), their centers must remain bounded. A convergent subsequence of centers would then yield a point every neighborhood of which has the same positive mass lower bound. Shrinking the neighborhood would give an atom, a contradiction. Thus \(r_0\) can be chosen uniformly, and \(f(r_0)\le\mu(\R^2)\) is also uniform. Since \(2\eps<1\), interchanging the order of the nonnegative integrals gives
\[
 \sup_{z_*}F_{z_*}(0)
 \le C_1\int_0^{r_0}t^{-2\eps}\dd t+\frac{\mu(\R^2)}{r_0}<\infty.
\]
Returning to \eqref{eq:massode}, we obtain
\begin{equation}\label{eq:ballarea}
 \mu(B_r(z_*))\le C_2r^2
 \quad\hbox{for every center and all sufficiently small radii}.
\end{equation}
The argument so far applies to radii with \(\mu(\partial B_r(z_*))=0\). Approximating any sufficiently small radius from below by such radii and using continuity from below removes this restriction. If \(E\) is a Lebesgue null set, then, for any \(\eta>0\), it has a covering by small balls
\(E\subset\bigcup_\ell B_{r_\ell}(z_\ell)\) with \(\sum_\ell r_\ell^2<\eta\). Hence
\[
 \mu(E)\le\sum_\ell\mu(B_{r_\ell}(z_\ell))
 \le C_2\sum_\ell r_\ell^2<C_2\eta.
\]
Letting \(\eta\downarrow0\) gives \(\mu\ll\dd z\). Write \(\mu=f_0\dd z\). By the Lebesgue differentiation theorem, for Lebesgue almost every \(z\),
\[
 f_0(z)=\lim_{r\downarrow0}\frac{\mu(B_r(z))}{\pi r^2}\le\frac{C_2}{\pi}.
\]

To ensure convergence at infinity, define the logarithmic potential
\[
 W(z)=-\frac1{2\pi}\int_{\R^2}
 \log\frac{|z-y|}{1+|y|}\dd\mu(y).
\]
This potential is well defined on compact sets: the local logarithmic singularity is integrable {when multiplied by} a bounded density, while at infinity the logarithm is \(O_K((1+|y|)^{-1})\) for \(z\) in a fixed compact set. Its gradient is locally bounded, since \(|z-y|^{-1}\) is locally integrable {when multiplied by} a bounded density and is bounded away from the singularity. In the sense of distributions, \(-\Delta W=f_0\). The singular kernel is now absolutely integrable. Split the numerator \(X(z)-X(y)\) of \(K_X\) in \eqref{eq:planarstress} into two terms, interchange \(z,y\) in the integral containing \(X(y)\), and substitute the integral formula for \(\nabla W\). This gives
\[
 -\int_{\R^2} f_0\diver X\dd z=2A\int_{\R^2} f_0\nabla W\cdot X\dd z
 \qquad(X\in C_c^\infty(\R^2;\R^2)).
\]

The left-hand side is the action of the distributional gradient \(\nabla f_0\) on \(X\), so
\(\nabla f_0=2Af_0\nabla W\). Since \(f_0\) is bounded and \(\nabla W\) is locally bounded,
we have \(f_0\in W^{1,\infty}_{\mathrm{loc}}\). The function \(W\) is also locally Lipschitz. The Sobolev product and chain rules therefore apply and give
\[
 \nabla(f_0e^{-2AW})
 =e^{-2AW}\nabla f_0-2Af_0e^{-2AW}\nabla W=0.
\]
A local Sobolev function with zero distributional gradient on the connected plane is constant. Thus \(f_0=C_0e^{2AW}\), and \(\mu\ne0\) implies \(C_0>0\).
Since \(-\Delta W=f_0\in L^\infty_{\mathrm{loc}}\), we have
\(W\in W^{2,q}_{\mathrm{loc}}\) for every finite \(q\). Taking \(q>2\) yields
\(W\in C^{1,\alpha}_{\mathrm{loc}}\). Combining this with \(f_0=C_0e^{2AW}\) and iterating Schauder estimates,
we find that \(W,f_0\) are smooth. The function
\(V=AW+\tfrac12\log(AC_0)\) then satisfies
\[
 -\Delta V=Af_0=AC_0e^{2AW}=e^{2V},
 \qquad \int_{\R^2}e^{2V}=A\mu(\R^2)<\infty.
\]
Theorem~\ref{tool:chenli} gives
\[
 f_0=\frac1A e^{2V}
 =\frac4A\frac{\lambda^2}{(1+\lambda^2|z-z_0|^2)^2},
\]
which is \eqref{eq:classdensity}.
\end{proof}

\subsection{Minimal-radius balls and the rescaled limit}
Retain the atom \(a_0\) and its mass \(p\) fixed at the beginning of this section, and the subsequence from Section~\ref{sec:pressure} along which \(\nu_j,\Pi_j\) converge weakly. Set \(A=4\pi/p\), passing to further subsequences as needed.

Choose \(\delta^{(n)}\downarrow0\) so that the boundaries of \(B_{\delta^{(n)}}(a_0)\) and \(B_{(\delta^{(n)})^2}(a_0)\) have zero \(\nu\) measure. Both families of balls decrease to \(\{a_0\}\), so continuity from above gives convergence of both masses to \(p\). By Proposition~\ref{prop:height}, the radii may also be chosen so that
\begin{align*}
 \limsup_j\left\|\Pi_j-\frac{\nu_j}A\right\|_{\TV(B_{\delta^{(n)}}(a_0))}&<\frac1n,\\
 \limsup_j\nu_j(B_{\delta^{(n)}}(a_0)\cap\{|s_j-A|>1/k\})&<\frac1n
 \quad(1\le k\le n).
\end{align*}
For fixed \(n\), weak convergence gives convergence of the masses of these two balls, since their boundaries have zero \(\nu\) measure. We can therefore choose strictly increasing indices \(j_n\) such that, for every \(j\ge j_n\),
\begin{align*}
 \left\|\Pi_j-\frac{\nu_j}A\right\|_{\TV(B_{\delta^{(n)}}(a_0))}&<\frac2n,\\
 \nu_j(B_{\delta^{(n)}}(a_0)\cap\{|s_j-A|>1/k\})&<\frac2n
 \quad(1\le k\le n),\\
 \left|\nu_j(B_{\delta^{(n)}}(a_0))-\nu(B_{\delta^{(n)}}(a_0))\right|&<\frac1n,\\
 \left|\nu_j(B_{(\delta^{(n)})^2}(a_0))-\nu(B_{(\delta^{(n)})^2}(a_0))\right|&<\frac1n.
\end{align*}
Set \(\delta_j=\delta^{(n)}\) for \(j_n\le j<j_{n+1}\). The corresponding index \(n\) tends to infinity with \(j\), and the preceding estimates give
\begin{equation}\label{eq:diagonal}
 \nu_j(B_{\delta_j}(a_0))=p+o(1),\quad
 \nu_j(B_{\delta_j^2}(a_0))=p+o(1),\quad
 \left\|\Pi_j-\frac{\nu_j}A\right\|_{\TV(B_{\delta_j}(a_0))}\to0,
\end{equation}
and, for every fixed \(\eps>0\), the \(\nu_j\) measure of \(B_{\delta_j}(a_0)\cap\{|s_j-A|>\eps\}\) tends to zero.

We next show that a ball of minimal radius can be chosen among those centered in \(\cl B_{\delta_j/2}(a_0)\) with \(\nu_j\) mass at least \(p/2\). Fix \(j\) large enough that \eqref{eq:diagonal} gives \(\nu_j(B_{\delta_j^2}(a_0))>p/2\). It suffices to consider radii in \([0,\delta_j^2]\). Define
\[
 \mathcal K_j:=\left\{(b,r)\in\cl B_{\delta_j/2}(a_0)\times[0,\delta_j^2]:
 \nu_j(B_r(b))\ge\frac p2\right\},
\]
with \(B_0(b)=\varnothing\). The set is nonempty because \((a_0,\delta_j^2)\in\mathcal K_j\).

Absolute continuity of \(\nu_j\) implies that \((b,r)\mapsto\nu_j(B_r(b))\) is continuous. Thus \(\mathcal K_j\) is a closed subset of the compact set \(\cl B_{\delta_j/2}(a_0)\times[0,\delta_j^2]\). The radius attains its minimum at some \((b_j,r_j)\in\mathcal K_j\), with
\[
 r_j=\min\{r:(b,r)\in\mathcal K_j\}\leq\delta_j^2.
\]
\rev{Since \(B_0(b)=\varnothing\), we have \(r_j>0\). If \(\nu_j(B_{r_j}(b_j))>p/2\), continuity in the radius would allow us to decrease \(r_j\) while keeping the mass at least \(p/2\), contradicting minimality. Hence \(\nu_j(B_{r_j}(b_j))=p/2\).}

For all sufficiently large \(j\), \(\delta_j/2+\delta_j^2<\delta_j\), so each minimal-radius ball lies in \(B_{\delta_j}(a_0)\). It must also intersect \(B_{\delta_j^2}(a_0)\). Otherwise, its \(\nu_j\) mass would be at most the difference of the two ball masses in \eqref{eq:diagonal}, namely \(o(1)\), contradicting its mass \(p/2\). Thus
\begin{equation}\label{eq:halfnormalization}
 |b_j-a_0|\le2\delta_j^2,\qquad r_j\le\delta_j^2,\qquad
 \nu_j(B_{r_j}(b_j))=p/2.
\end{equation}
In particular, \(b_j\) eventually lies in \(B_{\delta_j/2}(a_0)\). Write \(\widehat\Om_j=(\Om-b_j)/r_j\) for the rescaled domain and \(\mathcal S_j(x)=(x-b_j)/r_j\). Define \(\widehat\nu_j=(\mathcal S_j)_\#\nu_j\) and \(\widehat\Pi_j=(\mathcal S_j)_\#\Pi_j\), extending both by zero to \(\R^2\). More explicitly, after extending the original measures by zero outside \(\Om\), for every Borel set \(B\subset\R^2\),
\[
 \widehat\nu_j(B)=\nu_j(b_j+r_jB),\qquad
 \widehat\Pi_j(B)=\Pi_j(b_j+r_jB).
\]

For this choice of centers and radii, the limiting measure is determined as follows.
\begin{proposition}\label{prop:profile}
After passing to a subsequence,
\begin{equation}\label{eq:profile}
 \widehat\nu_j\weak\mu\quad\hbox{locally weakly on }\R^2,\qquad
 \dd\mu(z)=\frac{p}{\pi(1+|z|^2)^2}\dd z.
\end{equation}
Thus \(\mu(\R^2)=p\): the total mass of the rescaled limit equals the mass of the atom of \(\nu\) at \(a_0\).
\end{proposition}
\begin{proof}
Since \(\widehat\nu_j(\R^2)=1\), there is a locally weakly convergent subsequence \(\widehat\nu_j\weak\mu\). For every fixed \(T>0\), the ball \(b_j+r_jB_T\) eventually lies in \(B_{\delta_j}(a_0)\). The weak convergence inequality for open sets and \eqref{eq:diagonal} give \(\mu(B_T)\le p\). Letting \(T\to\infty\), we obtain \(\mu(\R^2)\le p\).
By \eqref{eq:halfnormalization} and the weak convergence inequality for closed sets,
\[
 \frac p2=\limsup_j\widehat\nu_j(\cl B_1)\le\mu(\cl B_1).
\]
For fixed \(z\in\R^2\) and \(0<t<1\), the center \(b_j+r_jz\) eventually belongs to \(\cl B_{\delta_j/2}(a_0)\). If \(\widehat\nu_j(B_t(z))>p/2\), then the original ball \(B_{tr_j}(b_j+r_jz)\) would be admissible with radius \(tr_j<r_j\), contradicting minimality of \(r_j\). Thus \(\widehat\nu_j(B_t(z))\le p/2\), and
\[
 \mu(B_t(z))\le\liminf_j\widehat\nu_j(B_t(z))\le p/2.
\]
Letting \(t\downarrow0\) shows that every atom of \(\mu\) has mass at most \(p/2\).

Fix \(X\in C_c^\infty(\R^2;\R^2)\), and use
\(X_j(x)=r_jX((x-b_j)/r_j)\) in \eqref{eq:stress}. The singular kernel is invariant under this scaling, the regular term is \(O(r_j)\), and
\(\diver X_j(x)=\diver X((x-b_j)/r_j)\). By \eqref{eq:diagonal}, replacing \(\widehat\Pi_j\) by \(\widehat\nu_j/A\) introduces an error in the divergence integral of at most
\[
 \tfrac12\|\diver X\|_\infty
 \left\|\Pi_j-\frac1A\nu_j\right\|_{\TV(B_{\delta_j}(a_0))}=o(1).
\]
Hence, for each fixed \(X\),
\[
 \frac1{4\pi}\iint_{\R^2\times\R^2} K_X(z,y)\dd\widehat\nu_j(z)\dd\widehat\nu_j(y)
 =\frac1{2A}\int_{\R^2}\diver X(z)\dd\widehat\nu_j(z)+o(1).
\]
Here \(\widehat\nu_j\) is extended by zero outside the rescaled domain \(\widehat\Om_j\).
If \(z\in\supp X\) and \(|y|\) is sufficiently large, then
\begin{equation}\label{eq:tailkernel}
 |K_X(z,y)|\le C_X/|y|.
\end{equation}
Thus the contribution from pairs with at least one variable outside \(B_L\) is at most \(C_X/L\), uniformly in \(j\).

If \(\mu(\{z_0\})=q>0\), take
\(X_\eps(z)=(z-z_0)\chi(|z-z_0|/\eps)\), with \(\chi\) as in Proposition~\ref{prop:square}.
Use the same estimates over the inner ball, annulus, and exterior as in that proposition, with \eqref{eq:tailkernel} controlling the contribution at infinity.
Letting successively \(j\to\infty\), the inner radius \(\eta\downarrow0\), and \(\eps\downarrow0\), we obtain
\[
 \frac{q^2}{4\pi}=\frac qA.
\]
This forces \(q=4\pi/A=p\), contrary to \(q\le p/2\). Hence \(\mu\) has no atoms.

It follows that
\[
 (\mu\otimes\mu)\{(z,y):z=y\}=\int\mu(\{z\})\dd\mu(z)=0.
\]
First insert fixed smooth compactly supported cutoffs in both variables. Approximation by finite sums of products of continuous functions shows that \(\widehat\nu_j\otimes\widehat\nu_j\) converges locally weakly to \(\mu\otimes\mu\). Within the cutoff supports, apply {the Portmanteau theorem} to the closed set \(\{|z-y|\le\eta\}\). Taking first the upper limit as \(j\to\infty\) and then letting \(\eta\downarrow0\), its mass tends to zero because the diagonal has zero \(\mu\otimes\mu\) measure. Since \(K_X\) is bounded, the corresponding integral also tends to zero.
Away from the diagonal, the kernel is continuous, so we may pass to the limit in the truncated integral. Finally, \eqref{eq:tailkernel} allows the cutoff radius to tend to infinity, giving \eqref{eq:planarstress}.

Lemma~\ref{lem:weakclassify} gives the density of \(\mu\) as
\((p/\pi)\lambda^2(1+\lambda^2|z-z_0|^2)^{-2}\). In particular,
\begin{equation}\label{eq:parentmass}
 \mu(\R^2)=p.
\end{equation}
Since \(\mu\) is absolutely continuous, every circle has zero \(\mu\) measure, and \eqref{eq:halfnormalization} gives \(\mu(B_1(0))=p/2\). We have already proved that \(\mu(B_t(z))\le p/2\) for every center \(z\) and \(0<t<1\). Letting \(t\uparrow1\) and using continuity from below gives \(\mu(B_1(z))\le p/2\).
{If \(z_0\ne0\), the sets \(B_1(z_0)\setminus B_1(0)\) and \(B_1(0)\setminus B_1(z_0)\) have equal positive area. The density is strictly larger than its value at distance one from \(z_0\) on the first set and strictly smaller on the second, up to their boundaries. Hence \(\mu(B_1(z_0))>\mu(B_1(0))=p/2\), contrary to the bound for unit balls. Thus \(z_0=0\).} Finally,
\(\mu(B_1)=p\lambda^2/(1+\lambda^2)=p/2\) implies \(\lambda=1\), proving \eqref{eq:profile}.
\end{proof}

\section{Selection of local maxima and neighborhood estimates}\label{sec:center}

The preceding section identified the limit of the rescaled measures. We now prove local smooth convergence of the corresponding functions, use it to select local maxima, and verify the integral conditions needed for radial comparison near these points.

\subsection{Strong convergence of the densities and parameter asymptotics}
Continue to rescale by the minimal radius \(r_j\), and define the circular mean
\[
 c_j=\frac1{2\pi}\int_0^{2\pi}u_j(b_j+r_je^{i\alpha})\dd\alpha.
\]
For \(z\in\widehat\Om_j\), set
\begin{equation}\label{eq:Wnormalization}
 W_j(z)=\frac{u_j(b_j+r_jz)-c_j}{M_j},\qquad
 A_j=M_jc_j,\qquad
 k_j=\frac{r_j^2\theta_j}{M_j^2}e^{c_j^2}.
\end{equation}
In the rescaled coordinates of this section, \(s_j(z)\) denotes the original function \(s_j(b_j+r_jz)\), so that \(s_j=A_j+M_j^2W_j\). A change of variables gives the density of the rescaled measure as
\[
 f_j(z):=\frac{\dd\widehat\nu_j}{\dd z}(z)
 =\frac{r_j^2}{M_j}\rho_j(b_j+r_jz).
\]
In the following equation, the Laplacian acts on the rescaled variable \(z\):
\begin{equation}\label{eq:fj}
 -\Delta W_j=f_j
 =k_js_je^{2A_jW_j+M_j^2W_j^2}
 =k_js_je^{(A_j+s_j)W_j}.
\end{equation}

\begin{proposition}\label{prop:center}
After passing to a subsequence,
\begin{align}
 &A_j\to A,\qquad k_j\to A^{-2},\qquad
 M_j^2\log(1/r_j)\to A^2/2,\label{eq:scalematch}\\
 &W_j\to W(z):=\frac1A\log\frac2{1+|z|^2}
 \quad\hbox{in }L^q_{\mathrm{loc}}(\R^2)\quad(1\le q<\infty),\label{eq:Wlimit}\\
 &f_j\to\frac{p}{\pi(1+|z|^2)^2}
 \quad\hbox{in }L^1_{\mathrm{loc}}(\R^2).\label{eq:strongsource}
\end{align}
\end{proposition}
\begin{proof}
In rescaled coordinates, subtract the {mean on the unit circle} from the Green representation formula for \(w_j=u_j/M_j\). The circular mean identity for the logarithm gives, on each fixed ball \(B_T\),
\begin{equation}\label{eq:Wpotential}
 W_j(z)=\frac1{2\pi}\int_{\mathbb{R}^2}
 \log\frac{\max\{1,|y|\}}{|z-y|}\dd\widehat\nu_j(y)+O_T(r_j).
\end{equation}
The error is uniform for \(z\in B_T\). Indeed, in the difference of the regular parts, both values of the first variable lie in a fixed compact subset of \(\Om\) and are \(O_T(r_j)\) apart, while \(\nu_j\) has unit total mass. As \(|y|\to\infty\), the kernel above is \(O_T(|y|^{-1})\). Since \(\widehat\nu_j(\R^2)=1\), the part of the integral over \(|y|>L\) has absolute value at most \(C_T/L\), uniformly in \(j\).
For \(y\) in a bounded set and every finite \(q\), the kernel depends continuously on \(y\) as an element of \(L^q(B_T)\). To see this, truncate the logarithmic singularity on \(|z-y|<\eta\). The truncated kernel is uniformly continuous, and the \(L^q\) norm of the truncation error is bounded by a constant multiple of
\((\int_0^\eta |\log t|^qt\dd t)^{1/q}\), independently of \(y\). Cover the bounded set of \(y\)-values by finitely many neighborhoods on each of which the \(L^q\) distance between kernels is less than a prescribed positive number, and choose a subordinate continuous partition of unity. Approximating by the kernels at the neighborhood centers gives a uniformly small \(L^q\) error. The integrals of the {functions in the partition of unity} converge by weak convergence of the measures. Thus the kernel integrals converge in \(L^q\). Finally, letting the outer cutoff radius tend to infinity gives \(W_j\to W\) in \(L^q(B_T)\), with \(W\) obtained by integrating the kernel in \eqref{eq:Wpotential} {with respect to} \(\mu\).
Since the density of \(\mu\) is radial, so is this integral. Its {mean on the unit circle} is zero,
as follows directly from the circular mean identity for the logarithmic kernel.
Writing \(W=W(R)\), we deduce from \(-\Delta W=\mu\) and \eqref{eq:profile} that
\[
 -2\pi R W'(R)=\mu(B_R)=\frac{pR^2}{1+R^2},
 \qquad W'(R)=-\frac2A\frac{R}{1+R^2}.
\]
Integration gives \(W(R)=C-A^{-1}\log(1+R^2)\). The condition \(W(1)=0\) then yields
\(C=A^{-1}\log2\), proving \eqref{eq:Wlimit}. Moreover,
the inequality \(|z-y|\le(1+T)\max\{1,|y|\}\) gives
\begin{equation}\label{eq:Wlower}
 W_j\ge-C_T\quad\hbox{on }B_T.
\end{equation}

Take a fixed ball of positive \(\mu\) measure. Its intersection with \(\{|s_j-A|>\eps\}\) has \(\widehat\nu_j\) measure tending to zero, so, for all sufficiently large \(j\), the ball contains a point where \(s_j\le A+\eps\). The identity \(s_j=A_j+M_j^2W_j\) and \eqref{eq:Wlower} give
\(0\le A_j\le A+\eps+C_TM_j^2\). Thus \(A_j\) is bounded.

We first prove that, for every fixed \(\alpha>0\) and \(T>0\),
\begin{equation}\label{eq:alllinear}
 \int_{B_T}e^{\alpha(W_j)_+}\dd z\le C_{T,\alpha}.
\end{equation}
Choose \(m>0\) with \(\alpha m<4\pi\). Since \(\mu\) is nonatomic, \(B_T\) can be covered by finitely many small balls \(B_\eta(z_0)\) {such that the corresponding closed balls of twice the radius have} \(\mu\) mass strictly less than \(m\). By the Portmanteau theorem, the corresponding \(\widehat\nu_j\) masses are also less than \(m\) for sufficiently large \(j\). On \(B_\eta(z_0)\), \eqref{eq:Wpotential} gives
\[
 W_j(z)\le C+\frac1{2\pi}\int_{B_{2\eta}(z_0)}
 \log\frac{C}{|z-y|}\dd\widehat\nu_j(y).
\]
{Applying Jensen's inequality to a probability measure constructed as in the local argument of} Lemma~\ref{lem:jensen}, we bound the exponential integral by
\[
 C|B_\eta|+\frac Cm\int_{B_{2\eta}(z_0)}
 \int_{B_\eta(z_0)}|z-y|^{-\alpha m/(2\pi)}\dd z\dd\widehat\nu_j(y).
\]
The inner integral is uniformly bounded because \(\alpha m/(2\pi)<2\). Summing over the finite cover proves \eqref{eq:alllinear}; replacing \(W_j\) by \((W_j)_+\) adds at most the area of the set.

Set \(\kappa_j=k_j(A_j+M_j^2)\). We first show that \(\kappa_j\) is bounded. Fix \(B_T\), choose \(C_T\) from \eqref{eq:Wlower}, and let \(L=2C_T+1\). For indices with \(A_j/M_j^2\le L\), we have \(c_j\le LM_j\), hence
\[
 \kappa_j=r_j^2\theta_j(A_j/M_j^2+1)e^{c_j^2}\le C.
\]
For the remaining indices, \(s_j\ge A_j/2\) throughout \(B_T\), and the exponent in \eqref{eq:fj} is bounded below by \(-2A_jC_T\). Since \(A_j+M_j^2\le C_LA_j\), we have \(f_j\ge c_T\kappa_j\) there. Integration over \(B_T\) gives \(c_T\kappa_j|B_T|\le\int_{B_T}f_j\le1\), proving boundedness of \(\kappa_j\).

Fix \(\eps>0\). On \(\{s_j\le A+\eps\}\), the density formula \eqref{eq:fj} gives, for all sufficiently large \(j\),
\begin{equation}\label{eq:truncdensity}
 f_j\ind_{\{s_j\le A+\eps\}}
 \le C\kappa_j(1+(W_j)_+)e^{(2A+3\eps)(W_j)_+}.
\end{equation}
For \(W_j\ge0\), use \(s_j=A_j+M_j^2W_j\le(A_j+M_j^2)(1+W_j)\), \(A_j\le A+2\eps\), and \(A_j+s_j\le2A+3\eps\). For \(W_j<0\), use \(s_j\le A_j\) and the fact that the exponential factor is at most one. By \eqref{eq:alllinear}, the truncated densities are uniformly bounded in \(L^q_{\mathrm{loc}}\) for every finite \(q\).

Fix a ball with \(\mu(B_T)>0\). Proposition~\ref{prop:height} gives \(\widehat\nu_j(B_T\cap\{s_j>A+\eps\})\to0\). Indeed, the ball corresponding to \(B_T\) in the original coordinates eventually lies in every fixed small ball about \(a_0\); let first \(j\to\infty\) and then the fixed radius tend to zero. For \(q>1\), let \(q'=q/(q-1)\) be the conjugate exponent. H\"older's inequality gives
\[
 \int_{B_T}|s_j-A_j|f_j\ind_{\{s_j\le A+\eps\}}\dd z
 \le M_j^2\|W_j\|_{L^{q'}(B_T)}
 \|f_j\ind_{\{s_j\le A+\eps\}}\|_{L^q(B_T)}\to0.
\]
Set \(\mathcal E_j=B_T\cap\{s_j\le A+\eps\}\). By Proposition~\ref{prop:height}, for every \(\eta>0\),
\[
 \int_{\mathcal E_j}|s_j-A|f_j
 \le\eta\widehat\nu_j(B_T)
 +(2A+\eps)\widehat\nu_j(B_T\cap\{|s_j-A|>\eta\}).
\]
The right-hand side tends to zero as first \(j\to\infty\) and then \(\eta\downarrow0\). Combining this with the preceding H\"older estimate gives
\[
 |A_j-A|\int_{\mathcal E_j}f_j
 \le\int_{\mathcal E_j}|s_j-A_j|f_j
    +\int_{\mathcal E_j}|s_j-A|f_j\to0.
\]
Since \(\int_{\mathcal E_j}f_j\to\mu(B_T)>0\), we conclude that \(A_j\to A\).

Boundedness of \(\kappa_j\) and \(A_j\to A>0\) imply boundedness of \(k_j\). Moreover, \(\liminf_j k_j>0\). Otherwise, a subsequence would satisfy \(k_j\to0\), hence \(\kappa_j\to0\); by \eqref{eq:truncdensity} and \eqref{eq:alllinear}, the \(\widehat\nu_j\) measure of \(B_T\cap\{s_j\le A+\eps\}\) would tend to zero, contradicting its positive limit.
Choose a subsequence with \(k_j\to k\in(0,\infty)\), and then a further one with \(W_j\to W\) almost everywhere on every compact set. Thus \(s_j=A_j+M_j^2W_j\to A\) almost everywhere. For fixed \(\eps>0\), we have \(\ind_{\{s_j\le A+\eps\}}\to1\) and \(f_j\ind_{\{s_j\le A+\eps\}}\to kAe^{2AW}\). A uniform \(L^q\) bound for some \(q>1\) ensures uniform integrability, so Vitali's convergence theorem gives local \(L^1\) convergence of the truncated densities. Since \(\int_{B_T\cap\{s_j>A+\eps\}}f_j\dd z\to0\), the densities \(f_j\) themselves converge locally in \(L^1\). Comparing with the density of \(\mu\) and the formula for \(W\), we find \(k=A^{-2}\), proving \eqref{eq:strongsource}.

Finally, take logarithms in the definition of \(k_j\) and multiply by \(M_j^2\):
\[
 M_j^2\log k_j
 =-2M_j^2\log(1/r_j)-M_j^2\log(1/\theta_j)
 -2M_j^2\log M_j+A_j^2.
\]
Proposition~\ref{prop:strict} gives \(M_j^2\log(1/\theta_j)\to0\). Also, \(M_j^2\log k_j\to0\) and \(M_j^2\log M_j\to0\). Using \(A_j\to A\), we obtain \(M_j^2\log(1/r_j)\to A^2/2\), as asserted in \eqref{eq:scalematch}.
\end{proof}

\subsection{Local upper bounds and selection of local maxima}\label{sec:representative}
To {strengthen the convergence} in Proposition~\ref{prop:center} to local \(C^2\) convergence, we need uniform upper bounds on disks. The following estimate will be used in Lemma~\ref{lem:diskheight} to extend an \(O(M_j^{-1})\) boundary bound to the interior.

\begin{lemma}\label{lem:lowdisk}
Let the disks \(D_j=B_{R_j}(c_j^*)\) lie in a fixed compact subset of \(\Om\). Suppose that \(u_j\) attains its maximum on \(\cl D_j\) at an interior point \(y_j\), and that
\[
 B_j=u_j(y_j)\to\infty,\qquad
 \sup_{\partial D_j}u_j=o(B_j),\qquad
 \log(1/R_j)=o(B_j^2),\qquad
 \log\theta_j=o(B_j^2).
\]
Assume the uniform upper bounds for the total mass and parameter in Lemma~\ref{lem:mass}. Then
\begin{equation}\label{eq:lowdiskmatch}
 \log\frac{\theta_jB_j^2}{4}
 =4\pi R(y_j)-1
 -B_j\int_{\Om\setminus B_{\exp(-B_j^2/4)}(y_j)}
 G(y_j,z)\rho_j(z)\dd z+o(1).
\end{equation}
\end{lemma}
\begin{proof}
Set \(\mu_j=(\theta_jB_j^2e^{B_j^2})^{-1/2}\) and \(d_j=\dist(y_j,\partial D_j)\).
Since \(\log\theta_j=o(B_j^2)\) and \(\log(1/R_j)=o(B_j^2)\),
\[
 \log\frac{R_j}{\mu_j}
 =\frac{B_j^2}2+\frac12\log\theta_j+\log B_j-\log(1/R_j)
 =\frac{B_j^2}2+o(B_j^2)\to\infty.
\]

After rescaling the disk by \(\mu_j\), the right-hand side of the equation for \(u_j/B_j\) satisfies
\[
 0\le\frac{\mu_j^2}{B_j}\rho_j
 =\frac1{B_j^2}\frac{u_j}{B_j}e^{u_j^2-B_j^2}\le B_j^{-2},
\]
because \(u_j\le B_j\) in \(D_j\).

We first prove that \(d_j/\mu_j\to\infty\). If this ratio were bounded along a subsequence, rescale \(u_j/B_j\) by \(\mu_j\) about the nearest boundary point to \(y_j\). The rescaled functions lie between zero and one, their boundary supremum is \(o(1)\), and the absolute value of the right-hand side is at most \(B_j^{-2}\). Subtract the harmonic extension of the boundary values. The resulting functions are bounded, vanish on the boundary, and have right-hand sides tending uniformly to zero. By the maximum principle, the subtracted harmonic functions have supremum \(o(1)\).
Since \(R_j/\mu_j\to\infty\), the rescaled boundaries converge locally smoothly to a line, so the constants in local boundary estimates can be chosen independently of \(j\). If the rescaled distance of the maximum point from the boundary tends to zero, these estimates force its value to tend to zero, contradicting the normalized maximum value one. If the distance tends to a finite positive number, a subsequence converges locally up to the boundary to a bounded harmonic function on a half-plane, with zero boundary values and value one at the limiting maximum point. Odd reflection across the boundary gives a bounded harmonic function on \(\R^2\). Liouville's theorem makes it constant, and the zero boundary values force it to vanish, again a contradiction. Hence \(d_j/\mu_j\to\infty\).

The functions
\(\eta_j(z)=B_j(u_j(y_j+\mu_jz)-B_j)\)
are now eventually defined on every fixed ball and satisfy
\(\eta_j\le0=\eta_j(0)\) and \(0\le-\Delta\eta_j\le1\).
The local \(W^{2,q}\) and Schauder estimates at the beginning of Section~\ref{sec:local} give a subsequence converging in \(C^2_{\mathrm{loc}}(\R^2)\).
Positivity of the Green function at the interior point \(y_j\), together with
\(\log(1/\mu_j)/B_j^2\to1/2\), shows as in Lemma~\ref{lem:maxbubble} that the limit \(\eta\) satisfies \(\int_{\R^2}e^{2\eta}\le4\pi\). Here \(y_j\) stays in a fixed compact subset of \(\Om\), and \(u_j\le B_j\) in \(D_j\), so the rescaled factor \((u_j/B_j)e^{u_j^2-B_j^2}\) is at most one. Theorem~\ref{tool:chenli} therefore gives \(\eta_j\to U\) in \(C^2_{\mathrm{loc}}\).

To estimate the distance to \(\partial D_j\), let \(h_j\) be the harmonic extension of \(u_j|_{\partial D_j}\) to \(D_j\). Subscripts indicate the domain: \(G_\Om=G\), while \(G_{D_j}\), \(H_{D_j}\), and \(R_{D_j}\) are the Dirichlet Green function of \(D_j\), its regular part with the sign convention \eqref{eq:green}, and its diagonal value. Subtracting the two Green representation formulas gives
\begin{equation}\label{eq:differencegreens}
 h_j(y_j)=\int_{D_j}(G_\Om-G_{D_j})(y_j,z)\rho_j(z)\dd z
 +\int_{\Om\setminus D_j}G_\Om(y_j,z)\rho_j(z)\dd z.
\end{equation}
By domain monotonicity and positivity of the Green function, every term on the right is nonnegative. With convention \eqref{eq:green}, the Robin function of the disk is
\[
 R_{D_j}(y_j)=-\frac1{2\pi}\log\bigl(d_j(2-d_j/R_j)\bigr).
\]
The explicit disk Green function also gives, for fixed \(T\), on \(B_{T\mu_j}(y_j)\),
\[
 H_{D_j}(y_j,z)=R_{D_j}(y_j)+O(T\mu_j/d_j)
 =R_{D_j}(y_j)+o(1).
\]
Indeed, after translating to a disk of radius \(R_j\) centered at the origin, its regular part is
\(-\frac1{2\pi}\log(|R_j^2-(y_j-c_j^*)\overline{(z-c_j^*)}|/R_j)\); 
The difference between this expression and its value at \(z=y_j\) is \(O(|z-y_j|/d_j)\).
The regular part for the original domain is bounded. In \eqref{eq:differencegreens}, retain only the integral over \(B_{T\mu_j}(y_j)\), with \(T\) fixed. The integral of \(\rho_j\) over this ball is \((m_T+o(1))/B_j\), so
\[
 \log(1/d_j)\le C B_j h_j(y_j)+C
 \le C B_j\sup_{\partial D_j}u_j+C=o(B_j^2).
\]
The same conclusion holds if \(d_j\) has a positive lower bound. For every fixed \(0<\sigma<1/2\),
\[
 \frac{e^{-\sigma B_j^2}}{d_j}
 =\exp\{-\sigma B_j^2+o(B_j^2)\}\to0,
\]
and hence \(B_{\exp(-\sigma B_j^2)}(y_j)\subset D_j\). On these balls, \(u_j\le B_j\), and we have already shown that the functions \(\eta_j\) obtained by rescaling about \(y_j\) by \(\mu_j\) converge to \(U\).
Apply the estimates from the end of the proof of Lemma~\ref{lem:maxbubble}. In the Green representation formula, first retain the integral over \(B_{T\mu_j}(y_j)\), and then use \(G(y_j,z)\ge\sigma B_j^2/(2\pi)-C\) on \(B_{\exp(-\sigma B_j^2)}(y_j)\setminus B_{T\mu_j}(y_j)\). The remaining integral is nonnegative and may be discarded. Letting first \(j\to\infty\) and then \(T\to\infty\) gives
\[
 B_j\int_{B_{\exp(-\sigma B_j^2)}(y_j)}\rho_j\to4\pi,
 \qquad \int_{B_{\exp(-\sigma B_j^2)}(y_j)}u_j\rho_j\to4\pi.
\]
Taking \(\sigma=1/32\) verifies the hypotheses of Lemma~\ref{lem:budget}. Lemma~\ref{lem:matching} with \(\tau=1/4\) then gives \eqref{eq:lowdiskmatch}.
\end{proof}

We now apply Lemma~\ref{lem:lowdisk} to obtain the interior bound needed to select local maximum points.
\begin{lemma}\label{lem:diskheight}
Let \(r_j\le\ell_j\to0\), and set \(D_j=B_{\ell_j}(b_j)\). Suppose that there are fixed constants \(c_0,C_0>0\) such that
\[
 c_0\le M_j^2\log(1/\ell_j)\le C_0,
 \qquad \sup_{\partial D_j}u_j\le C_0/M_j,
\]
Then \(\sup_{D_j}u_j\le C/M_j\).
\end{lemma}
\begin{proof}
Suppose, to the contrary, that a subsequence satisfies \(M_jB_j\to\infty\), where \(B_j:=\max_{\cl D_j}u_j\). The maximum is attained at an interior point \(y_j\), and
\[
 \frac{\log(1/\ell_j)}{B_j^2}
 \le\frac{C_0}{(M_jB_j)^2}\to0,
 \qquad
 \frac{\log(1/\theta_j)}{B_j^2}
 =\frac{M_j^2\log(1/\theta_j)}{(M_jB_j)^2}\to0.
\]
The boundary values are \(o(B_j)\), and the disks eventually lie in a fixed compact subset of \(\Om\). Lemma~\ref{lem:lowdisk} therefore applies.

Set \(h_j=e^{-B_j^2/4}\). {By \eqref{eq:scalematch},
\[
 M_j^2\log\frac{h_j}{r_j}
 =-\frac{(M_jB_j)^2}{4}+M_j^2\log\frac1{r_j}
 \longrightarrow-\infty,
\]
so \(h_j/r_j\to0\).} Local weak convergence to the nonatomic measure \(\mu\) implies
\begin{equation}\label{eq:smallballuniform}
 \sup_{z\in K}\widehat\nu_j(B_{\eta_j}(z))\to0
 \quad\hbox{for compact }K\subset\R^2\hbox{ and }\eta_j\to0.
\end{equation}
{Indeed, if the supremum did not tend to zero, there would be \(c>0\) and centers \(z_j\in K\), converging along a subsequence to \(z\), such that \(\widehat\nu_j(B_{\eta_j}(z_j))\ge c\). For every fixed \(t>0\), these balls eventually lie in \(\cl B_t(z)\). Portmanteau's inequality gives
\[
 c\le\limsup_j\widehat\nu_j(\cl B_t(z))\le\mu(\cl B_t(z)).
\]
Letting \(t\downarrow0\) contradicts \(\mu(\{z\})=0\).}

Fix \(T\) with \(\mu(B_T)>3p/4\). The ball \(B_{h_j}(y_j)\) in the original coordinates corresponds to \(B_{h_j/r_j}((y_j-b_j)/r_j)\). If this ball intersects \(B_T\), its center eventually belongs to \(\cl B_{T+1}\), so \eqref{eq:smallballuniform} bounds the \(\widehat\nu_j\) mass of the intersection by \(o(1)\). Thus
\[
 \nu_j(B_{Tr_j}(b_j)\setminus B_{h_j}(y_j))\ge p/2
\]
for all sufficiently large \(j\). Since \(r_j\le\ell_j\) and \(y_j\in D_j\), every point of \(B_{Tr_j}(b_j)\setminus B_{h_j}(y_j)\) is at distance at most \((T+1)\ell_j\) from \(y_j\). Write \(I_{\mathrm{out},j}=\int_{\Om\setminus B_{h_j}(y_j)}G(y_j,z)\rho_j(z)\dd z\). Restricting the integral to this set gives
\[
 I_{\mathrm{out},j}\ge\frac{pM_j}{2}
 \left(\frac1{2\pi}\log\frac1{(T+1)\ell_j}-C\right)
 \ge\frac c{M_j}.
\]
By \eqref{eq:lowdiskmatch} and local boundedness of the Robin function,
\[
 \frac{cB_j}{M_j}\le B_jI_{\mathrm{out},j}
 \le C+\log(1/\theta_j).
\]
Multiplication by \(M_j^2\) contradicts \(M_jB_j\to\infty\) and \eqref{eq:strict}.
\end{proof}

\begin{proposition}\label{prop:representative}
The functions \(W_j,W\) of Proposition~\ref{prop:center} satisfy
\[
 W_j\to W\quad\hbox{in }C^2_{\mathrm{loc}}(\R^2).
\]
There exist local maximum points \(x_j\to a_0\) of \(u_j\), with \(\beta_j=u_j(x_j)\), such that
\begin{equation}\label{eq:representative}
 M_j\beta_j\to A,\qquad
 \beta_j(u_j(x_j+\mu_jz)-\beta_j)\to U(z)
 \quad\hbox{in }C^2_{\mathrm{loc}}(\R^2),\qquad
 \mu_j=(\theta_j\beta_j^2e^{\beta_j^2})^{-1/2}.
\end{equation}
Moreover,
\begin{equation}\label{eq:scaleconstants}
 \frac{x_j-b_j}{r_j}\to0,\qquad \frac{\mu_j}{r_j}\to\frac12,\qquad
 \lim_{T\to\infty}\lim_{j\to\infty}
 \nu_j(B_{T\mu_j}(x_j))=p.
\end{equation}
\end{proposition}
\begin{proof}
We first estimate the oscillation of \(u_j\) on circles. For \(B_{R_1}(b_j)\Subset\Om\) and \(0<R_0<R_1\), integration in polar coordinates gives
\[
 \int_{R_0}^{R_1}\int_0^{2\pi}
 |\partial_tu_j(b_j+Re^{it})|^2\dd t\,\frac{\dd R}{R}\le E_j.
\]
Hence there exists \(R\in[R_0,R_1]\) for which Cauchy--Schwarz on the circle yields
\begin{equation}\label{eq:circle-oscillation}
 \mathop{\rm osc}_{\partial B_R(b_j)}u_j
 \le\sqrt{2\pi}\left(\int_0^{2\pi}|\partial_tu_j(b_j+Re^{it})|^2\dd t\right)^{1/2}
 \le\left(\frac{2\pi E_j}{\log(R_1/R_0)}\right)^{1/2}.
\end{equation}

Fix \(T>0\), choose \(2+T<R_0<R_1\), and apply this estimate to the original radii in \([r_jR_0,r_jR_1]\). Since \(M_j^2E_j\to0\), we can select \(t_j\in[R_0,R_1]\) such that the oscillation on \(\partial B_{r_jt_j}(b_j)\) is \(o(M_j^{-1})\).
On the other hand, taking the circular mean in \eqref{eq:Wpotential} gives
\[
 \frac1{2\pi R}\int_{|z|=R}W_j\dd\sigma
 =\frac1{2\pi}\int
 \log\frac{\max\{1,|y|\}}{\max\{R,|y|\}}\dd\widehat\nu_j(y)+O_{R_1}(r_j).
\]
The absolute value of the logarithm is at most \(|\log R|\), so the mean of \(u_j\) on the selected circle is \(c_j+O(M_j)\). Set \(D_j=B_{r_jt_j}(b_j)\). Then
\begin{equation}\label{eq:circlelow}
 \sup_{\partial D_j}u_j\le\frac{A+o(1)}{M_j}.
\end{equation}
By \eqref{eq:scalematch}, \(M_j^2\log(1/(r_jt_j))\to A^2/2\). Lemma~\ref{lem:diskheight} therefore gives a uniform upper bound for \(s_j\) on \(B_T\).

Return to \eqref{eq:fj}. On a fixed slightly larger compact set, we have \(s_j\le C\), \(A_j\le C\), \(k_j\le C\), and \(W_j\ge-C\), hence
\[
 f_j\le C e^{C(W_j)_+}.
\]
By \eqref{eq:alllinear}, the functions \(f_j\) are uniformly bounded in \(L^q_{\mathrm{loc}}\) for every finite \(q\). Interior \(W^{2,q}\) estimates for \(-\Delta W_j=f_j\), together with the \(L^q\) bounds for \(W_j\), yield local \(C^{1,\alpha}\) bounds when \(q>2\). The right-hand side of \eqref{eq:fj} then has uniform \(C^{0,\alpha}\) bounds on smaller compact sets, and Schauder estimates give \(C^2\) compactness. The known \(L^q\) convergence identifies every subsequential limit with \(W\); hence \(W_j\to W\) locally in \(C^2\).

The limit \(W\) has a unique strict maximum at the origin. Maximize \(W_j\) on a fixed small closed ball about the origin. Uniform convergence ensures that each such maximum point \(z_j\) is eventually interior and that \(z_j\to0\). Set \(x_j=b_j+r_jz_j\) and \(\beta_j=u_j(x_j)\). Then \(x_j\) is a local maximum point, and
\[
 M_j\beta_j=A_j+M_j^2W_j(z_j)\to A.
\]
Since \(M_j\to0\) and \(A>0\), we have \(\beta_j\to\infty\).

Furthermore,
\[
 r_j^2\theta_j\beta_j^2e^{\beta_j^2}
 =k_j(M_j\beta_j)^2 e^{\beta_j^2-c_j^2}
 \longrightarrow A^{-2}A^2 e^{2AW(0)}=4.
\]
Here \(\beta_j^2-c_j^2=(M_j\beta_j+A_j)W_j(z_j)\to2AW(0)\), and \(AW(0)=\log2\). Thus \(\mu_j/r_j\to1/2\), and
\[
 \beta_j\bigl(u_j(x_j+\mu_jz)-\beta_j\bigr)
 =(M_j\beta_j)\bigl(W_j(z_j+(\mu_j/r_j)z)-W_j(z_j)\bigr)\to U(z)
\]
in \(C^2_{\mathrm{loc}}\). For each fixed \(T>0\), changing variables on the original ball \(B_{T\mu_j}(x_j)\) gives \(\beta_j\int_{B_{T\mu_j}(x_j)}\rho_j\to\int_{B_T}e^{2U}\). Dividing by \(M_j\beta_j\to A\) and then letting \(T\to\infty\) proves the last assertion of \eqref{eq:scaleconstants}.
\end{proof}

\subsection{Estimates on larger neighborhoods}\label{sec:neighborhood}
Proposition~\ref{prop:representative} gives convergence after rescaling about the selected local maximum points. To apply the local radial comparison from Section~\ref{sec:local}, we also need to control the weighted integral over balls with radii larger than \(\mu_j\). The following neighborhood estimates supply the remaining integral and parameter bounds, verifying all the hypotheses of Lemma~\ref{lem:budget}.

\begin{proposition}\label{prop:neighborhood}
\rev{Fix \(0<\sigma<A^2/2\), and set \(\varrho_j=e^{-\sigma/M_j^2}\). Then}
\begin{equation}\label{eq:neighborhood}
 \sup_{B_{4\rev{\varrho_j}}(b_j)}u_j\le C_\sigma/M_j,\qquad
 \nu_j(B_{2\rev{\varrho_j}}(b_j))\to p,\qquad
 \int_{B_{2\rev{\varrho_j}}(b_j)}u_j\rho_j\dd x\to4\pi.
\end{equation}
In particular, the local maximum points selected in Proposition~\ref{prop:representative} satisfy all the hypotheses of Lemma~\ref{lem:budget}.
\end{proposition}
{\begin{proof}
By \eqref{eq:scalematch}, \(r_j/\rev{\varrho_j}\to0\). Applying \eqref{eq:circle-oscillation} to radii in \([4\rev{\varrho_j},8\rev{\varrho_j}]\), we can choose \(\ell_j\in[4\rev{\varrho_j},8\rev{\varrho_j}]\) such that
\[
 \mathop{\rm osc}_{\partial B_{\ell_j}(b_j)}u_j
 \le\left(\frac{2\pi E_j}{\log2}\right)^{1/2}
 =o(M_j^{-1}),
\]
where the last equality follows from \(M_j^2E_j\to0\) in \eqref{eq:strict}.
To estimate the mean on this circle, set
\[
 \overline u_j(\ell_j)
 :=\frac1{2\pi\ell_j}\int_{\partial B_{\ell_j}(b_j)}u_j\dd\sigma
 =\frac1{2\pi}\int_0^{2\pi}u_j(b_j+\ell_je^{it})\dd t.
\]
Choose a constant \(C\ge\max\{L_0,1\}\), depending only on \(\Om\). The Green representation formula and Tonelli's theorem, applied to the nonnegative integrand \(G(b_j+\ell_je^{it},y)\rho_j(y)\), give
\[
 \begin{aligned}
 \overline u_j(\ell_j)
 &=\int_\Om\left[\frac1{2\pi}\int_0^{2\pi}
 G(b_j+\ell_je^{it},y)\dd t\right]\rho_j(y)\dd y\\
 &\le\frac1{2\pi}\int_\Om\left[\frac1{2\pi}\int_0^{2\pi}
 \log\frac{C}{|b_j+\ell_je^{it}-y|}\dd t\right]\rho_j(y)\dd y,
 \end{aligned}
\]
where we used \eqref{eq:greenbounds}. For each fixed \(y\), the circular mean identity for the logarithm yields
\[
 \frac1{2\pi}\int_0^{2\pi}\log|b_j+\ell_je^{it}-y|\dd t
 =\log\max\{\ell_j,|b_j-y|\}.
\]
It follows that
\[
 \begin{aligned}
 \overline u_j(\ell_j)
 &\le\frac1{2\pi}\int_\Om
 \log\frac{C}{\max\{\ell_j,|b_j-y|\}}\rho_j(y)\dd y\\
 &\le\frac{M_j}{2\pi}\log\frac{C}{\ell_j}
 \le\frac{M_j}{2\pi}\log\frac{C}{\varrho_j}\\
 &=\frac{\sigma}{2\pi M_j}+\frac{M_j\log C}{2\pi}
 =O_\sigma(M_j^{-1}).
 \end{aligned}
\]
 Combining this mean bound with the oscillation estimate gives
\[
 \sup_{\partial B_{\ell_j}(b_j)}u_j
 \le\overline u_j(\ell_j)+\mathop{\rm osc}_{\partial B_{\ell_j}(b_j)}u_j
 =O_\sigma(M_j^{-1}).
\]
Moreover, \(4\le\ell_j/\varrho_j\le8\), so
\[
 M_j^2\log\frac1{\ell_j}
 =\sigma-M_j^2\log\frac{\ell_j}{\varrho_j}\longrightarrow\sigma>0.
\]
Lemma~\ref{lem:diskheight} therefore gives \(\sup_{B_{4\rev{\varrho_j}}(b_j)}u_j\le C_\sigma/M_j\).

Write \(m_T=\int_{B_T}e^{2U}\uparrow4\pi\). For fixed \(T\), the relations \(r_j/\varrho_j\to0\), \(|x_j-b_j|=o(r_j)\), and \(\mu_j/r_j\to1/2\) give
\(B_{T\mu_j}(x_j)\subset B_{2\varrho_j}(b_j)\) for all sufficiently large \(j\). The convergence after rescaling about \(x_j\) and \(M_j\beta_j\to A\) yield
\[
 \liminf_{j\to\infty}\nu_j(B_{2\varrho_j}(b_j))
 \ge\lim_{j\to\infty}\nu_j(B_{T\mu_j}(x_j))
 =\frac{m_T}{A}.
\]
Letting \(T\to\infty\) gives the lower bound \(4\pi/A=p\). Conversely, for every fixed sufficiently small \(t>0\) with \(\nu(\partial B_t(a_0))=0\), the ball \(B_{2\varrho_j}(b_j)\) eventually lies in \(B_t(a_0)\). Hence
\[
 \limsup_{j\to\infty}\nu_j(B_{2\varrho_j}(b_j))\le\nu(B_t(a_0)).
\]
Letting \(t\downarrow0\) proves \(\nu_j(B_{2\varrho_j}(b_j))\to p\).

Fix \(T>0\). The same rescaling gives
    \begin{align}\label{eq:m_T}
    	&\int_{B_{T\mu_j}(x_j)}u_j\rho_j\dd x
    	=\int_{B_{T\mu_j}(x_j)}\theta_ju_j^2e^{u_j^2}\dd x
    	\longrightarrow m_T,\\
    	&\nu_j(B_{T\mu_j}(x_j))
    	=\frac1{M_j}\int_{B_{T\mu_j}(x_j)}\rho_j\dd x
    	\longrightarrow\frac{m_T}{A}.
    \end{align}
The pointwise bound established above gives
\[
 0\le s_j=M_ju_j\le C_\sigma
 \quad\hbox{on }B_{2\varrho_j}(b_j),
\]
Since \(u_j\rho_j\dd x=s_j\dd\nu_j\), we have
    \begin{equation}\label{eq:int_B_2rho_j-b_ju_jrho_j}
    	\int_{B_{2\varrho_j}(b_j)}u_j\rho_j\dd x=\int_{B_{T\mu_j}(x_j)}u_j\rho_j\dd x+\int_{B_{2\varrho_j}(b_j)\setminus B_{T\mu_j}(x_j)}s_j\dd\nu_j.
    \end{equation}
The remaining integral satisfies
\[
 \begin{aligned}
 0&\le\int_{B_{2\varrho_j}(b_j)\setminus B_{T\mu_j}(x_j)}s_j\dd\nu_j\\
 &\le C_\sigma\bigl[\nu_j(B_{2\varrho_j}(b_j))
 -\nu_j(B_{T\mu_j}(x_j))\bigr].
 \end{aligned}
\]
The mass convergence proved above and the {limit on the inner ball for fixed} \(T\) give
    \begin{equation}\label{eq:mass-converge-to-p-m_T/A}
    	\nu_j(B_{2\varrho_j}(b_j))-\nu_j(B_{T\mu_j}(x_j))
    	\longrightarrow p-\frac{m_T}{A}.
    \end{equation}
Combining \eqref{eq:m_T}, \eqref{eq:int_B_2rho_j-b_ju_jrho_j}, and \eqref{eq:mass-converge-to-p-m_T/A}, we obtain
\[
 \begin{aligned}
 m_T&\le\liminf_{j\to\infty}\int_{B_{2\varrho_j}(b_j)}u_j\rho_j\dd x\\
 &\le\limsup_{j\to\infty}\int_{B_{2\varrho_j}(b_j)}u_j\rho_j\dd x
 \le m_T+C_\sigma\left(p-\frac{m_T}{A}\right).
 \end{aligned}
\]
Finally, \(m_T\uparrow4\pi=Ap\), and \(C_\sigma\) does not depend on \(T\), so the error term \(C_\sigma(p-m_T/A)\) tends to zero as \(T\to\infty\). Letting \(T\to\infty\) in these bounds gives
\[
 \int_{B_{2\varrho_j}(b_j)}u_j\rho_j\dd x\longrightarrow4\pi,
\]
which is the last limit in \eqref{eq:neighborhood}.

To verify the hypotheses of Lemma~\ref{lem:budget}, fix \(0<\sigma<A^2/64\). Since \(M_j\beta_j\to A\),
\[
 \frac{e^{-\beta_j^2/32}}{\rev{\varrho_j}}
 =\exp\left\{-\frac{(M_j\beta_j)^2/32-\sigma}{M_j^2}\right\}\to0.
\]
Also, \(|x_j-b_j|=o(r_j)=o(\rev{\varrho_j})\), so \(B_{\exp(-\beta_j^2/32)}(x_j)\subset B_{2\rev{\varrho_j}}(b_j)\) for all sufficiently large \(j\).
For fixed \(T\), we have \(T\mu_j/e^{-\beta_j^2/32}\to0\). Thus
\[
 \int_{B_{T\mu_j}(x_j)}u_j\rho_j
 \le\int_{B_{\exp(-\beta_j^2/32)}(x_j)}u_j\rho_j
 \le\int_{B_{2\rev{\varrho_j}}(b_j)}u_j\rho_j.
\]
Letting first \(j\to\infty\) and then \(T\to\infty\), both bounds tend to \(4\pi\).
Moreover,
\[
 \frac{|\log\theta_j|}{\beta_j^2}
 =\frac{M_j^2\log(1/\theta_j)}{(M_j\beta_j)^2}\to0.
\]
By Proposition~\ref{prop:representative}, the selected local maxima lie in a fixed compact subset of \(\Om\), and the functions rescaled about them by \(\mu_j\) converge to \(U\). Together with the uniform bounds on the total mass and parameter, this verifies all hypotheses of Lemma~\ref{lem:budget}.
\end{proof}}

\section{Proof of the energy bound}\label{sec:absorb}\label{sec:mainproof}

The preceding section provides the required local maximum points and estimates in their neighborhoods. We now control the Green potential integral outside small balls about these points and deduce \(M_j=O(\sqrt{\theta_j})\). Combined with the maximum bound from Section~\ref{sec:local}, this excludes divergence of the energy.

We work along the subsequence associated with a fixed atom. Constants may depend on that atom and subsequence, but are independent of \(j\).

\begin{proposition}\label{prop:absorb}
For the local maximum points \(x_j\) selected in Proposition~\ref{prop:representative}, set
\[
 r_j^\circ=e^{-3\beta_j^2/16},\qquad
 I_{\mathrm{out},j}^\circ
 =\int_{\Om\setminus B_{r_j^\circ}(x_j)}
 G(x_j,y)\rho_j(y)\dd y.
\]
Then, for all sufficiently large \(j\),
\begin{equation}\label{eq:absorbconclusion}
 I_{\mathrm{out},j}^\circ\le C M_j,\qquad
 \theta_j\beta_j^2\ge c>0,\qquad
 M_j/\sqrt{\theta_j}\le C.
\end{equation}
\end{proposition}
\begin{proof}
We suppress the subscript \(j\) when no confusion can arise and write
\[
 r=e^{-3\beta^2/16},\quad d=e^{-\beta^2/16},\quad
 Q=\int_{B_r(x_j)}\rho,\quad P_r=\int_{B_r(x_j)}P.
\]

\emph{The error estimate from the Pohozaev identity.}
For this calculation, translate \(x_j\) to the origin. The vector field
\[
 \mathcal T(x)=\frac{xP(x)}2+(x\cdot\nabla u)\nabla u-\frac x2|\nabla u|^2
\]
\rev{satisfies \(\diver\mathcal T=P\), as follows directly from \(-\Delta u=\rho\) and \(\nabla P=2\rho\nabla u\).}
On \(\partial B_r\),
\(\mathcal T\cdot n=\frac r2[P+(\partial_nu)^2-(\partial_\tau u)^2]\).
and therefore
\[
 P_r=\frac r2\int_{\partial B_r}
 [(\partial_nu)^2-(\partial_\tau u)^2+P]\dd\sigma.
\]
Integrating the equation gives the mean normal derivative
\(\frac1{2\pi r}\int_{\partial B_r}\partial_nu\,\mathrm d\sigma=-Q/(2\pi r)\). Decompose the normal derivative into its circular mean and a mean-zero term to obtain
\begin{equation}\label{eq:defect}
 D:=P_r-\frac{Q^2}{4\pi}
 =\frac r2\int_{\partial B_r}
 \left[\left(\partial_nu+\frac Q{2\pi r}\right)^2
 -(\partial_\tau u)^2+P\right]\dd\sigma.
\end{equation}
By Proposition~\ref{prop:neighborhood} and local comparison, the gradient of \(u-B_j\) is at most \(C/(\beta d)\) on \(B_d\). On each circle, the radial function \(B_j\) has constant normal derivative and zero tangential derivative.

Writing \(v=u-B_j\) and using the fact that the mean of \(\partial_nu\) on \(\partial B_r\) is \(-Q/(2\pi r)\), we have
\[
 \partial_nu+\frac{Q}{2\pi r}
 =\partial_nv-\frac1{2\pi r}\int_{\partial B_r}\partial_nv\dd\sigma,
 \qquad \partial_\tau u=\partial_\tau v.
\]
Thus the difference between \(\partial_nu\) and its circular mean has absolute value at most \(2\|\nabla v\|_\infty\), while \(|\partial_\tau u|\) is at most
\(\|\nabla v\|_\infty\). Hence
\[
 \left|\partial_nu+\frac Q{2\pi r}\right|\le\frac{2C}{\beta d},
 \qquad |\partial_\tau u|\le\frac C{\beta d}.
\]
Comparison of the function values and \eqref{eq:meanexpansion} with \(\tau=3/16\) give \(u=(3/8+o(1))\beta\) uniformly on \(\partial B_r\). Therefore
\begin{equation}\label{eq:defectbound}
 \frac{|D|}{r}\le\frac{Cr}{\beta^2d^2}
 +Cr\theta\exp\bigl(\max_{\partial B_r}u^2\bigr)
 =o(M^2).
\end{equation}
The first term is \(Ce^{-\beta^2/16}/\beta^2\). The exponent in the second is
\((-3/16+9/64+o(1))\beta^2=(-3/64+o(1))\beta^2\); 
and \(\theta\) is bounded. Both terms are \(o(\beta^{-2})\), while \(M\beta\to A>0\).
The choice \(3/16\in(2/16,1/4)\) ensures exponential decay of both error terms.

\emph{A test vector field from the gradient of a convex function.}
Fix \(\delta>0\) sufficiently small, and choose \(\chi\in C_c^\infty(B_{2\delta}(a_0))\) with \(0\le\chi\le1\) and \(\chi=1\) on \(B_\delta(a_0)\). Its support lies in \(\Om\). For all sufficiently large \(j\), \(|x_j-a_0|<\delta/4\) and \(r<\delta/4\).
Set
\[
 Y(x)=\frac{x-x_j}{\max\{r,|x-x_j|\}},\qquad X=\chi Y.
\]
The vector field \(Y\) is the gradient of the convex function
\[
 \Phi_r(x)=
 \begin{cases}
 |x-x_j|^2/(2r),&|x-x_j|\le r,\\
 |x-x_j|-r/2,&|x-x_j|\ge r.
 \end{cases}
\]
so \(K_Y\ge0\). The identity
\[
 K_{\chi Y}(x,y)=\frac{\chi(x)+\chi(y)}2K_Y(x,y)
 +\frac{\chi(x)-\chi(y)}2
 \frac{(Y(x)+Y(y))\cdot(x-y)}{|x-y|^2}
\]
and \(|Y|\le1\) give \(K_X\ge-\Lip\chi\).
On \(B_r(x_j)^2\), the kernel equals \(1/r\).

For \(0<\eps<1/4\), set
\[
 Q_\eps=\int_{B_{\eps r}(x_j)}\rho\dd x,\qquad
 J=\int_{B_\delta(a_0)\setminus B_r(x_j)}
 \frac{\rho(y)}{|y-x_j|}\dd y.
\]
Translate \(x_j\) to the origin and take \(|x|\le\eps r\), \(t=|y|\ge r\). A direct calculation gives
\[
 (Y(x)-Y(y))\cdot(x-y)
 =|x|^2/r+t-x\cdot y(1/r+1/t)\ge t(1-2\eps),
 \qquad |x-y|^2\le t^2(1+\eps)^2.
\]
Thus \(K_Y(x,y)\ge(1-4\eps)/t\).
Retain the pairs in \(B_r(x_j)\times B_r(x_j)\), together with the pairs between
\(B_{\eps r}(x_j)\) and \(B_\delta(a_0)\setminus B_r(x_j)\) in both orders.
These three product regions are pairwise disjoint. Bounding the kernel on the remainder below by \(-\Lip\chi\), we obtain
\begin{equation}\label{eq:interactionlower}
 \frac1{4\pi}\iint K_X\rho(x)\rho(y)\dd x\dd y
 \ge\frac{Q^2}{4\pi r}+\frac{(1-4\eps)Q_\eps}{2\pi}J-C_\delta M^2.
\end{equation}

For each fixed \(j\), Lemma~\ref{lem:stress} applies to this compactly supported Lipschitz vector field. The regular Green term has absolute value at most \(C_\delta M^2\). Almost everywhere, the divergence satisfies
\[
 \diver X=
 \begin{cases}
 2/r,&x\in B_r(x_j),\\
 |x-x_j|^{-1},&x\in B_\delta(a_0)\setminus B_r(x_j),
 \end{cases}
\]
and its absolute value is at most \(C_\delta\) on the cutoff annulus
\(B_{2\delta}(a_0)\setminus B_\delta(a_0)\).
By \eqref{eq:pressuretotal}, \(\int P=O(M^2)\). Thus, after multiplying \eqref{eq:stress} by \(M^2\), its right-hand side becomes
\[
 \frac{P_r}{r}
 +\frac12\int_{B_\delta(a_0)\setminus B_r(x_j)}
 \frac{P(y)}{|y-x_j|}\dd y+O_\delta(M^2).
\]
Combining this with \eqref{eq:interactionlower} and writing the difference of the two terms associated with \(B_r(x_j)\) as \(D/r\), we obtain
\begin{equation}\label{eq:absorbmaster}
 \frac{(1-4\eps)Q_\eps}{2\pi}J
 \le\frac Dr+
 \frac12\int_{B_\delta(a_0)\setminus B_r(x_j)}
 \frac{P(y)}{|y-x_j|}\dd y+C_\delta M^2.
\end{equation}
For fixed \(\eps>0\), we have \(\eps r/\mu_j\to\infty\). Convergence of the densities after rescaling by \(\mu_j\) gives
\(\liminf_j\beta_jQ_\eps\ge4\pi\). Since \(Q_\eps\le Q\) and \eqref{eq:flux} gives \(\beta_jQ\to4\pi\), the upper limit is at most \(4\pi\). Hence
\begin{equation}\label{eq:Qeps}
 \beta_jQ_\eps\to4\pi.
\end{equation}

\emph{Splitting the weighted integral by function values.}
On \(\{u\ge\beta/3\}\), \eqref{eq:pressure-comparison} gives \(P\le3\rho/\beta\), so the contribution to the integral involving \(P\) in \eqref{eq:absorbmaster} is at most \(3J/(2\beta)\).

For the integral over \(\{u_j<\beta_j/3\}\), take
\[
 T=3A/8,\qquad q=5/2,\qquad q'=5/3.
\]
Since \(M_j\beta_j\to A\), the inequality \(u_j<\beta_j/3\) eventually implies \(s_j\le T\). Decrease the fixed \(\delta\) and choose
\(\nu(\cl B_{2\delta}(a_0))<\bar p<16p/15\). The identity \(Ap=4\pi\) then gives \(qT\bar p<4\pi\). Applying \eqref{eq:localjensen} and the truncation estimate \eqref{eq:lowheight-Lq}, we obtain
\begin{equation}\label{eq:lowLq}
 \left\|\ind_{\{u_j<\beta_j/3\}}\frac{P_j}{M_j^2}
 \right\|_{L^q(B_\delta(a_0))}\le C_\delta\theta_j.
\end{equation}
Since \(q'<2\) and \(|x_j-a_0|<\delta/4\),
\[
 \int_{B_\delta(a_0)}|y-x_j|^{-q'}\dd y
 \le\frac{2\pi}{2-q'}(5\delta/4)^{2-q'}.
\]
H\"older's inequality therefore bounds the integral of \(P_j\), weighted by \(|y-x_j|^{-1}\), over \(\{u_j<\beta_j/3\}\) by \(C_\delta M_j^2\theta_j\).

Fix \(\eps=1/100\). By \eqref{eq:Qeps}, the coefficient on the left of \eqref{eq:absorbmaster} is eventually at least \(9/(5\beta_j)\). The coefficient of \(J\) contributed by \(\{u_j\ge\beta_j/3\}\) is at most \(3/(2\beta_j)\). Moving that term to the left leaves a coefficient of at least \(3/(10\beta_j)\). Using \eqref{eq:defectbound} and boundedness of \(\theta_j\), we obtain
\[
 \frac3{10\beta_j}J
 \le |D|/r+C_\delta M_j^2+C_\delta M_j^2\theta_j
 \le C_\delta M_j^2.
\]
Thus \(J\le C_\delta\beta_jM_j^2\le C_\delta M_j\).

\emph{The outer integral and the parameter estimate.}
Outside \(B_\delta(a_0)\), the integration variable is at distance at least \(3\delta/4\) from \(x_j\), so this contribution to \(I_{\mathrm{out},j}^\circ\) is at most \(C_\delta M_j\). On \(B_\delta(a_0)\setminus B_r(x_j)\), the logarithmic upper bound for the Green function gives
\(G(x_j,y)\le C_\delta/|x_j-y|\), so the corresponding integral is at most \(C_\delta J\).
Hence \(I_{\mathrm{out},j}^\circ\le C_\delta M_j\). Take \(\tau=3/16\) in \eqref{eq:matching}. Since \(M_j\beta_j\to A\) and \(R(x_j)\) is bounded,
\[
 \log(\theta_j\beta_j^2/4)\ge-C,
\]
which proves \(\theta_j\beta_j^2\ge c>0\). Finally,
\[
 \frac{M_j^2}{\theta_j}
 =\frac{(M_j\beta_j)^2}{\theta_j\beta_j^2}\le C.
\]
This completes the proof.
\end{proof}

\begin{proof}[Proof of Theorem~\ref{thm:main}]
Suppose that the theorem fails, and choose positive solutions with \(E_j\ge j\). Lemma~\ref{lem:mass} and \eqref{eq:energyidentity} give \(M_j\le C_\Om\) and \(\gamma_j\to\infty\). Proposition~\ref{prop:maxbound} yields \(\theta_j\to0\) and
\(\sqrt{\theta_j}\gamma_j\le C_\Om\).

As in Section~\ref{sec:pressure}, pass to a subsequence along which \(\nu_j,\Pi_j\) both converge weakly. Proposition~\ref{prop:atom} shows that \(\nu\) has an atom of positive mass. Proposition~\ref{prop:representative} then provides local maximum points \(x_j\) associated with this atom. By Proposition~\ref{prop:absorb}, a further subsequence satisfies \(M_j/\sqrt{\theta_j}\le C\). Therefore
\[
 E_j\le\gamma_jM_j
 =(\sqrt{\theta_j}\gamma_j)(M_j/\sqrt{\theta_j})\le C.
\]
The constant may depend on the fixed atom and subsequence, but is independent of \(j\). This contradicts \(E_j\ge j\).

There is therefore a finite bound \(B_\Om\), depending only on the fixed domain \(\Om\), for which \eqref{eq:main} holds.
\end{proof}

\section*{Declarations}

\noindent\textbf{Funding.} W. Zou is supported by the National Key R\&D Program of China (Grant No.~2023YFA1010001) and the National Natural Science Foundation of China (Grant No.~12171265).

\noindent\textbf{Conflict of interest.} The authors declare no conflict of interest related to this work.

\noindent\textbf{Data availability.} This is a study in theoretical mathematics. No datasets were generated or analyzed, and data sharing is not applicable.

\noindent\textbf{Use of generative artificial intelligence.} \rev{During research discussions and preparation of the manuscript, generative artificial intelligence was used to assist with literature searches and reference checks, discussion and elaboration of mathematical arguments, language revision, and \LaTeX{} typesetting. The authors assume full responsibility for the mathematical validity, accuracy, and integrity of the proofs presented in the manuscript.}

\end{document}